%% file: main.tex
\documentclass[11pt]{amsart}
\usepackage[margin=1.25in]{geometry}

\usepackage{amsmath}
\usepackage{upgreek}
\usepackage{amssymb}
\usepackage{amsthm}
\usepackage{tikz}
\usepackage{tikz-cd}
\usepackage{spectralsequences}
\usetikzlibrary{arrows, patterns, calc}
\usepackage{mathrsfs}
\usepackage{hyperref}
\usepackage{bm}
\usepackage[foot]{amsaddr}
\usepackage[capitalise]{cleveref}

\newcommand{\newrefformat}[2]{}

\crefname{lemma}{Lemma}{Lemmas}
\crefname{theorem}{Theorem}{Theorems}
\crefname{definition}{Definition}{Definitions}
\crefname{proposition}{Proposition}{Propositions}
\crefname{remark}{Remark}{Remarks}
\crefname{corollary}{Corollary}{Corollaries}
\crefname{equation}{Equation}{Equations}
\crefname{ex}{Example}{Examples}
\crefname{appsec}{Appendix}{Appendices}
\theoremstyle{plain}
\newtheorem*{theorem*}{Theorem}
\newtheorem{theorem}{Theorem}[section]
\newtheorem{lemma}[theorem]{Lemma}
\newtheorem{corollary}[theorem]{Corollary}
\newtheorem{proposition}[theorem]{Proposition}

\makeatletter
\let\c@equation\c@theorem
\makeatother

\theoremstyle{definition}
\newtheorem{definition}[theorem]{Definition}
\newtheorem{example}[theorem]{Example}
\newtheorem{assumption}[theorem]{Assumption}
\newtheorem{construction}[theorem]{Construction}

\theoremstyle{remark}
\newtheorem{remark}[theorem]{Remark}

\numberwithin{equation}{section}

\newcommand{\M}{\mathcal{M}}
\newcommand{\R}{\mathbb{R}}

\newcommand{\RR}{\mathbb{R}}
\newcommand{\QQ}{\mathbb{Q}}
\newcommand{\NN}{\mathbb{N}}
\newcommand{\ZZ}{\mathbb{Z}}

\newcommand{\abs}[1]{\lvert#1\rvert}

\newcommand{\gen}[1]{\langle #1 \rangle}

\newcommand{\cat}[1]{\mathscr{#1}}

\newcommand{\grad}{\nabla}
\newcommand{\codim}{{\rm codim}}

\newcommand{\crit}{{\rm Crit}}
\newcommand{\bndry}{\partial}
\newcommand{\im}{\operatorname{im}}
\newcommand{\op}{\operatorname{op}}
\newcommand{\sd}{\operatorname{sd}}
\newcommand{\colim}{\operatorname{colim}}

\newcommand{\nd}{\text{nd}}

\newcommand{\Cat}{{\rm Cat}}
\newcommand{\Top}{{\rm Top}}
\newcommand{\Set}{{\rm Set}}

\newcommand{\ob}{\operatorname{Ob}}
\renewcommand{\hom}{\operatorname{Hom}}

\newcommand{\id}{\operatorname{id}}
\newcommand{\tw}{\operatorname{tw}}
\newcommand{\sk}{\operatorname{sk}}

\definecolor{dark0}
{rgb}{0.2666, 0.003921,0.3294}
\definecolor{dark1}
{rgb}{0.28235,0.15686, 0.470588}
\definecolor{dark2}
{rgb}{0.2844, 0.3155, 0.60444}
\definecolor{dark3}
{rgb}{0.19216, 0.4078, 0.55686}
\definecolor{mid0}
{rgb}{0.149, 0.5098, 0.55686}
\definecolor{mid1}
{rgb}{0.13777, 0.70222, 0.608888}
\definecolor{mid2}
{rgb}{0.1216, 0.6196, 0.537}
\definecolor{mid3}
{rgb}{0.2078, 0.7176, 0.4745}
\definecolor{mid4}
{rgb}{0.427, 0.8039, 0.349}
\definecolor{mid5}
{rgb}{0.70588, 0.8706, 0.1725}
\definecolor{light}
{rgb}{0.992, 0.90588, 0.145}
\usepackage{comment}
\usepackage[textwidth=25mm, textsize=tiny]{todonotes}

\author[M. E.\ Calle and F.\ Liu]{Maxine E. Calle and Fangji Liu}
\title{On the classifying space of a Morse flow category}

\address{Mathematical Institute, University of Bonn\\Endenicher Allee 60, 53115 Bonn\\Germany}
\email{callem@math.uni-bonn.de}

\address{Department of Mathematics, University of Pennsylvania, Philadelphia, PA 19104}
\email{fjliu@sas.upenn.edu}

\date{\today}



\begin{document}
\maketitle

\begin{abstract}
    We show that the classifying space of the flow category of a \emph{tame} Morse function on a smooth, closed manifold $M$ recovers the homotopy type of $M$, thereby addressing a claim in a preprint of Cohen--Jones--Segal. The tameness assumption is that the compactified moduli spaces of broken gradient trajectories are locally contractible, ensuring the flow category is topologically well-behaved. We construct a Morse function and Riemannian metric on $S^2\times S^1$ for which the associated flow category fails to recover the correct homotopy type, showing that the tameness hypothesis is crucial. Together, these results  clarify the extent to which transversality assumptions can be relaxed so that the flow category models the homotopy type of the underlying manifold.
\end{abstract}

\begin{figure}[h!]
    \centering
    \includegraphics[width=0.5\linewidth]{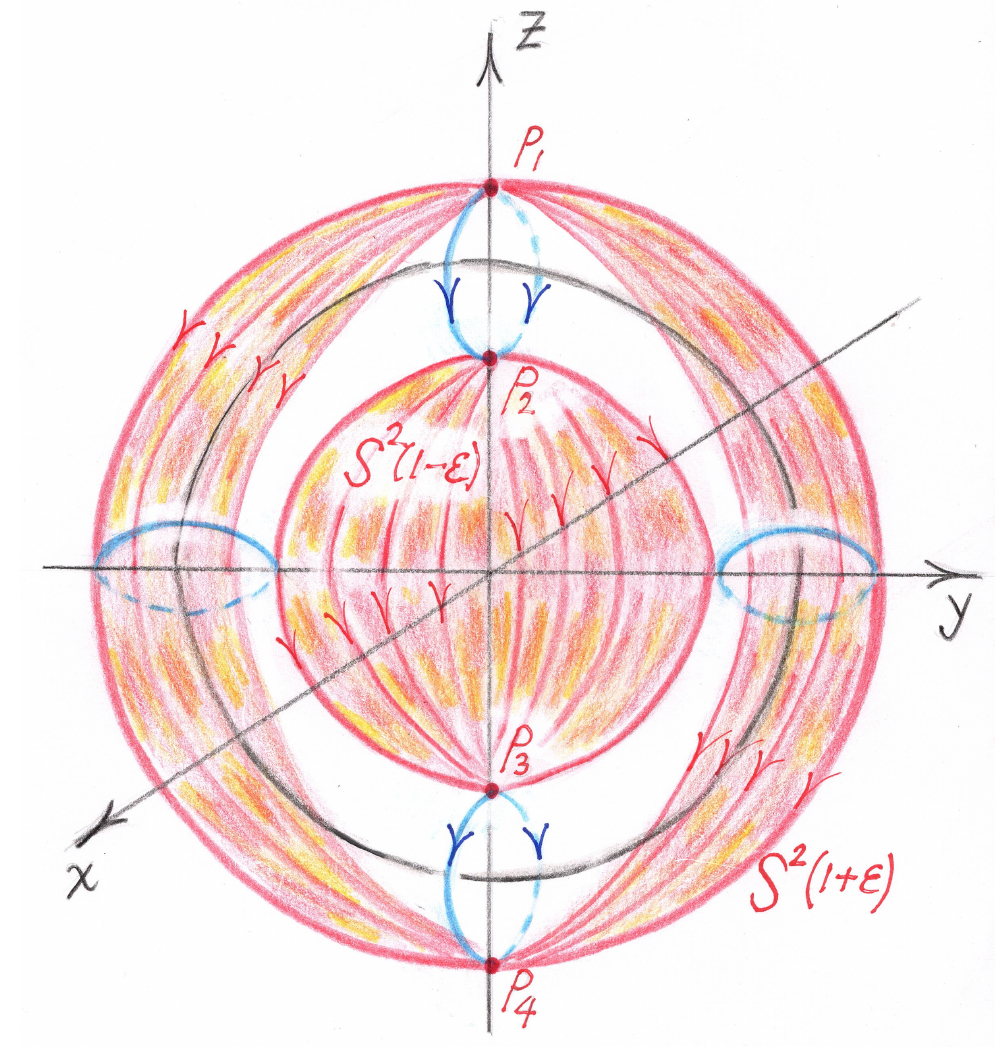}
    \caption{Sketch of the negative gradient flow of the $z$-coordinate height function on the boundary of an $\varepsilon$-tubular neighborhood in $\RR^4$ of the unit $S^2$ in $\RR^3$. The $w$-axis points out of the page towards the reader, and the blue circles are solid in the positive $w$-direction and dashed in the negative $w$-direction. Sketch provided to the authors by Herman Gluck.}
    \label{fig:Herman}
\end{figure}
\section{Introduction}

\noindent Given a Morse function $f\colon M\to \RR$ on a smooth, closed, finite-dimensional Riemannian manifold $(M,g)$, the topology of $M$ is reflected in the structure of its critical points and the gradient flow lines connecting them.
In their influential preprint \cite{cohen/jones/segal:1995}, Cohen--Jones--Segal introduced the \emph{flow category} $\cat C_f$, a topological category whose objects are the critical points of $f$ and whose morphism spaces are the compactified moduli spaces of gradient flow lines between them. This construction provides a categorical model for Morse theory and serves as a precursor to analogous constructions in Floer theory \cite{CJS:floer, cohen2007, abouzaidblumberg:2024}.

One of the central claims of \cite{cohen/jones/segal:1995} is that when $(f,g)$ is Morse--Smale, the classifying space $B\cat C_f$ is homeomorphic to $M$. Although the preprint was never published, for reasons discussed in \cite{cohen:2019}, subsequent work established the foundational analytic results required to justify this statement, including the structure of compactified moduli spaces as manifolds with corners and the associativity of gluing \cite{wehrheim:2012, qin:2011, large:21}. A complete proof in the Morse--Smale case is given in \cite[\S 12.2]{cohen/iga/norbury:2006}.

Beyond the Morse--Smale setting, Cohen--Jones--Segal further assert that for an arbitrary Morse function there is a homotopy equivalence
\[
B\cat C_f \simeq M.
\]
The argument given in \cite{cohen/jones/segal:1995}, however, contains a gap (see \cref{rmk:CJS gap}), which essentially is due to the fact that the assignment of a point $p\in M$ to the unbroken gradient flow line $\gamma_p$ passing through $p$ is not continuous in general.

Our first main theorem identifies a natural topological condition on the compactified moduli spaces under which the claimed homotopy equivalence does hold.

\begin{theorem}[\cref{thm:loc}]\label{introthm:CJS}
   Let $(M,g)$ be a smooth, closed, finite-dimensional Riemannian manifold. If $f\colon M\to \RR$ is a tame Morse function, then there is a homotopy equivalence $B\cat C_f\simeq M$.
\end{theorem}

The tame assumption (\cref{def:tame}) is that the compactified moduli spaces of gradient flows are locally contractible, and hence ensures that these moduli spaces cannot be too pathological. This condition identifies the topological regularity needed in the absence of transversality and is satisfied in many natural examples, including classical Morse functions which fail to be Morse--Smale such as the vertical torus (see \cref{ex:torus}).

The tameness hypothesis is crucial. In \cref{counter}, we construct a Morse function on $S^2\times S^1$ and a Riemannian metric for which the associated flow category fails to recover the homotopy type of the manifold.

\begin{theorem}[\cref{thm:counter}]\label{introthm:counter}
    There is a Morse function $f\colon S^2\times S^1\to \RR$ and a Riemannian metric $g$ on $S^2\times S^1$ such that $B\cat C_f$ is not weakly homotopy equivalent to $S^2\times S^1$.
\end{theorem}

The example arises from a general procedure, described in \cref{sec:twisting}, which produces compactified moduli spaces of broken trajectories that fail to be locally contractible. In fact, \cref{thm:twisting} shows that \textit{any} compact subset $Z\subseteq \RR$ may be realized as an (uncompactified) moduli space of gradient flow lines of a suitable Morse function on a $3$-manifold. Applying this construction for $Z=\{1/n\mid n\geq 1\}\cup\{0\}$ in the concrete case of the height function on $S^2\times S^1\subseteq \RR^4$ yields the counterexample above. 

The question of how much topological regularity is required of compactified moduli spaces in order to recover homotopy-theoretic information arises naturally in Morse and Floer theory. Classical Morse theory packages the gradient flow of a Morse--Smale function into a chain complex whose homology recovers the singular homology of $M$. 
\cref{introthm:CJS} shows that the homotopy type of $M$, not just its homology, is already present in the flow data for a larger class of Morse functions. 

\cref{introthm:CJS} is also the finite-dimensional, unstable prototype for the paradigm of Floer homotopy theory. 
The original motivation of Cohen--Jones--Segal was to upgrade Floer homology theories from graded groups to (stable) homotopy types \cite{CJS:floer}. In that context, one constructs a ``Floer flow category'' and from it obtains a spectrum whose homology is the given Floer homology. These flow categories often have additional structure that is utilized to construct the associated spectrum, although the precise extra structure can depend on the geometric context. For instance, in the Morse--Smale setting, one can use the manifolds with corners structure on the moduli spaces to refine $\cat C_f$ to a \textit{framed flow category}. There is then a general procedure (explained, for instance, in \cite[\S 5]{bonciocat}) that produces a spectrum from a framed flow category, and one can show that this spectrum associated to $\cat C_f$ is stably equivalent to $\Sigma^{\infty}_+ M$ \cite{CBK24, PS24, ADP25, bonciocat, blakey}. In other words, there is a stable analogue of \cref{introthm:CJS} in the Morse--Smale setting.  More generally, Cohen--Jones--Segal's perspective has continued to shape a wide range of developments in Floer homotopy theory, a complete account of which is far too extensive to include here but see e.g.\ \cite{manolescu:03, kraugh, lipshitzsarkar, lawson/lipshitz/sarkar, abouzaidblumberg:21}.

While the original Cohen--Jones--Segal construction is usually presented under transversality assumptions, in practice many natural gradient-like flows --- even for nice Morse functions --- are not strictly Morse–Smale, or one wants to work in contexts (with families, equivariance, or constraints) where enforcing transversality is awkward or impossible \cite{wasserman, CBK24, rezchikov:2024}.  
Together, \cref{introthm:CJS,introthm:counter} isolate a precise, weaker condition under which the gradient flow information still captures the homotopy type of the underlying manifold. 

\subsection{Outline} We begin in \cref{sec:MT} with some preliminaries on Morse theory, and the compactified moduli spaces of gradient flow lines are discussed in \cref{sec:moduli}. \cref{sec:thm} introduces the flow category and the proof of \cref{introthm:CJS} is in \cref{sec:pf}. In \cref{counter}, we describe a general strategy for producing pathological moduli spaces and the specific counterexample from \cref{introthm:counter} is analyzed in \cref{sec:counterex}. There are two appendices: \cref{sec:topcats} provides some background on topological categories and simplicial spaces and 
\cref{app:top} shows that various topologies on the compactified moduli spaces that appear in the literature coincide.
\subsection*{Acknowledgements} 
We would like to thank Andrew Blumberg, Nir Gadish, and Mona Merling for helpful feedback on an early draft, as well as Herman Gluck for his aid in visualizing $S^2\times S^1$ and for providing \cref{fig:Herman}. We are also grateful to the anonymous referees for their comments which greatly improved the paper.
Additionally, the first author would like to thank Ralph Cohen, who generously shared many insights about \cite{cohen/jones/segal:1995}; Kyle Ormsby, who first introduced her to \cite{cohen/jones/segal:1995} and supervised her bachelor's thesis \cite{calle:2020}; as well as David Ayala and Ang\'elica Osorno for insights which led to the completion of \cite{calle:2020}. 
The first author was partially supported by NSF grant DGE-1845298, and is now supported by the Deutsche Forschungsgemeinschaft (DFG, German Research Foundation) under Germany's Excellence Strategy - GZ 2047/2, Projekt-ID 390685813.
\section{Preliminaries on Morse theory}\label{sec:MT}
In this section, we summarize some important concepts from Morse theory, particularly relating to the moduli spaces of gradient flow lines. For more detailed exposition and proofs, see \cite{milnor:1963,audin/damian:2014, cohen/iga/norbury:2006}. We assume throughout that $M$ is a smooth, closed, finite-dimensional Riemannian manifold.

\subsection{Morse theory and gradient flow}\label{subsec:mt and gf} Given a smooth function $f\colon M\to \RR$, recall that the \textit{critical points} of $f$ are those points $p\in M$ such that $(df)_p=0$. We denote the collection of critical points of $f$ by $\crit(f)$. 

\begin{definition}
    A function $f\colon M\to \RR$ is \textit{Morse} if all its critical points are non-degenerate. The \textit{Morse index} $\mu(p)$ of a critical point $p$ is the index of $(d^2f)_p$, that is, the maximum dimension of a subspace upon which the Hessian at $p$ is negative-definite.
\end{definition}

The Morse index classifies different critical points based on the local behavior of $M$, and can be counted as the number of negative entries in the diagonalization of $(d^2f)_p$. The Morse index also completely determines the behavior of $f$ at this point, via the Morse Lemma, see \cite[Lemma 2.2]{milnor:1963}. An immediate corollary of the Morse Lemma is that critical points of a Morse function are isolated, and in particular a Morse function on a compact manifold can only have finitely many critical points.

The \textit{gradient flow lines} are the integral curves of the vector field $-\grad f$ with respect to a chosen Riemannian metric $g$ on $M$. The function $f$ is non-increasing along gradient flow lines, and is strictly decreasing along a gradient flow line that does not contain a critical point (cf.\ \cite[Lemma 4.5]{cohen/iga/norbury:2006}).
Note that if the image of a gradient flow line $\gamma$ contains a critical point $p$, then in fact it must be the constant curve: $\gamma(t)\equiv p$ for all $t$. Thus there are two kinds of flow lines: the constant flow lines at critical points and the flow lines that stay away from critical points (but may get arbitrarily close) along which $f$ is strictly decreasing.

\begin{definition}\label{def:minimal flow}
    For each $x\in M$, the \textit{minimal flow line} for $x$ is the unique gradient flow $\gamma_x$ through $x$.
\end{definition}

The existence and uniqueness of minimal flow lines is standard from the theory of ordinary differential equations, see e.g.\ \cite[Theorem 4.6]{cohen/iga/norbury:2006}. Another standard aspect of the gradient flow lines is that all the flow lines ``start'' and ``end''  at critical points, in the sense that for any gradient flow line $\gamma\colon \RR\to M$, there exist critical points $s(\gamma), t(\gamma)\in \crit(f)$ such that\[
\lim_{t\to -\infty} \gamma(t)=s(\gamma) ~\text{ and }~ \lim_{t\to \infty} \gamma(t)=e(\gamma)
\] 

\begin{definition}
    We say that $s(\gamma)$ is the \textit{starting point} and $t(\gamma)$ is the \textit{ending point} of $\gamma$. 
\end{definition}

Fixing a critical point $p$, the stable (resp. unstable) manifold of $p$ is the collection of all points whose minimal flow line ends at $p$ (resp. those points whose minimal flow line starts at $p$).

\begin{definition}\label{1defn:unstable manifold}  The \textit{stable manifold} of $p\in \crit(f)$ is \[
W^s(p) = \{x\in M \mid t( \gamma_x)=p\},
\]
and the \textit{unstable manifold} of $p$ is \[
W^u(p) = \{ x\in M \mid s( \gamma_x)=p\}.
\]
\end{definition} 

The stable and unstable manifolds of $p\in \crit(f)$ are submanifolds of $M$ that are diffeomorphic to open disks (cf.\ \cite[Proposition 2.1.5]{audin/damian:2014}), with
\[
\dim(W^u(p)) = \codim(W^s(p))=\mu(p).
\]
The pair $(f, g)$ is said to be \textit{Morse--Smale} if $
W^u(a)\pitchfork W^s(b)$ is a transverse intersection for all $p,q\in \crit(f)$. In \cite{smale:1961}, Smale shows that Morse--Smale pairs exist and are dense (see also \cite[\S 2.2.c]{audin/damian:2014}).

\begin{definition}
    Let $W(p,q)=W^u(p) \cap W^s(q)$ be all points of $M$ on the trajectories connecting $p$ to $q$, i.e. \[
W(p,q) = \{ x \in M\mid  s(\gamma_x)=p \text{ and } t(\gamma_x)=q\}.
\] The group $\RR$ acts on $W(p,q)$ by translations in time, and this action is free
when $p\neq q$. Consequently, we can consider \textit{moduli space of flow lines} which is the quotient
\[
\mathcal{M}(p,q):= W(p,q) / \RR.
\] When $p=q$, we take $\mathcal M(p,p) = \{*\}$, which we think of as the constant curve at $p$.
\end{definition}

In the Morse--Smale case, $W(p,q)$ is a $(\mu(p)-\mu(q))$-dimensional submanifold of $M$ and consequently $\mathcal{M}(p,q)$ is a manifold of dimension $\mu(p)-\mu(q)-1$ (for $p\neq q$). Note then that the Morse--Smale condition also prohibits flow lines between points with the same Morse index. However, outside of the Morse--Smale case, $W(p,q)$ may be poorly behaved, as we will later demonstrate in \cref{sec:counterex}. 

\subsection{The compactified moduli space of flow lines}\label{sec:moduli}
The flow lines provide a partial ordering on the set of critical points, with $p \succeq q$ if $W(p,q)$ is non-empty, or equivalently, if there is a flow line $\gamma$ that starts at $p$ and ends at $q$. We say $p\succ q$ if $p\succeq q$ and $p\neq q$, and we call a sequence of critical points ${\bf p}=\{p_1,\dots, p_k\}$ \textit{ordered} if $p_i\succeq p_{i+1}$ for all $i$. We will often work with ordered sequences ${\bf p}=\{p, p_1,\dots, p_k, q\}$ which go from $p\in \crit(f)$ to some other $q\in \crit(f)$, and we say such a sequence has \textit{length} $k$.

\begin{definition}\label{2defn:sp of broken flow}
For $p,q\in \crit(f)$, define\begin{align*}
    \overline {\mathcal M}(p,q) = \coprod_{{\bf p}} \left(\mathcal{M}(p,p_1) \times\dots \times \mathcal{M}(p_k, q)\right)/\sim,
\end{align*}
where the disjoint union is over ordered sequences of critical points ${\bf p} = \{p, p_1,\dots, p_k, q\}$ and $\sim$ is the equivalence relation generated by \[(\gamma_1, \dots, \gamma_i, \id_{p_i}, \gamma_{i+1},\dots, \gamma_{n})\sim (\gamma_1, \dots, \gamma_i, \gamma_{i+1}, \dots, \gamma_n).\] The curves in $\overline{\mathcal{M}}(p,q)$ are thus smooth on $M\setminus\crit(f)$, hence referred to as \textit{broken flow lines}; the curves in $\mathcal{M}(p,q)$ are said to be \textit{unbroken}. 

We put a topology on $\overline{\M}(p,q)$, induced by the following metric: for two broken gradient flow lines $\gamma,\eta$, define
\[ d(\gamma,\eta) := \sup_{x\in \im \gamma}\inf_{y\in \im \eta} d_M(x,y) + \sup_{y\in \im \eta}\inf_{x\in \im \gamma} d_M(x,y) \]
where $d_M$ is the distance function on $M$; i.e.\ $d(\gamma, \eta)$ is the Gromov--Hausdorff distance between $\im \gamma$ and $\im \eta$. Note that whenever $p=q$, the only possible flow line is the steady solution, so $\overline{\mathcal M}(p,p) = \{\id_p\}$. The topological space $\overline{\M}(p,q)$ is called the \textit{moduli space of broken flow lines} from $p$ to $q$.
\end{definition}

The topology on $\overline{\M}(p,q)$ may also be described in several equivalent ways, as in \cite[Section 1]{cohen/jones/segal:1995}, \cite[Definition 4.10]{BH:Morse-Bott}, or \cite[Section 3.2.a]{audin/damian:2014}; for definitions and a comparison of these topologies, see \cref{app:top}.
As is suggested by the notation, this space is a compactification of $\mathcal{M}(p,q)$ \cite[\S I.3.2]{audin/damian:2014}. We make the following additional observations about the point-set topological properties of the compactified moduli spaces.

\begin{lemma}\label{cor:sep}
    For all $p,q\in \crit(f)$, the moduli space $\overline{\mathcal M}(p,q)$ is compact Hausdorff. Moreover, there is an embedding $\overline{\M}(p,q)\hookrightarrow \R^N$ for some $N$, and consequently $\overline{\M}(p,q)$ is second countable.
\end{lemma}\begin{proof}
    Compactness is proved in \cite[\S 3.2.b]{audin/damian:2014}, and $\overline{\M}(p,q)$ is Hausdorff because it is a metric space (see Appendix \ref{app:top}).
    For the second claim, suppose that $a=f(p)$, $b=f(q)$, and that the critical values of $f$ between $b$ and $a$ are $b=a_0<a_1<a_2<\cdots<a_k=a$. Let $c_i=(a_{i-1}+a_i)/2$ for $1\leq i\leq k$ and $M_i=f^{-1}(c_i)\subseteq M$. Since $c_i$ is a regular value of $f$, $M_i$ is a codimension one closed submanifold of $M$. Now the map
    \[ \overline{\M}(p,q)\to M_1\times M_2\times \cdots \times M_k,\qquad \gamma\mapsto (\im\gamma\cap M_1,\cdots,\im\gamma\cap M_k) \]
    is continuous and injective and hence an embedding because $\overline{\M}(p,q)$ is compact, and $M_1\times\cdots \times M_k$ is Hausdorff. Since each $M_i$ can be embedded in Euclidean space (by the Whitney embedding theorem), it follows that $\overline{\M}(p,q)$ can be embedded in some $\R^N$.
\end{proof}

\begin{remark}\label{rmk:height parametrized flows}
An element of $\overline{\mathcal M}(p,q)$ is a tuple of flow lines $(\gamma_0,\dots, \gamma_k)$ so that there is an equality $\lim_{t\to \infty} \gamma_i(t) = \lim_{t\to -\infty}\gamma_{i+1}(t)$ (for each $0\leq i\leq k$), however we often think of this tuple as the ``composition'' $\gamma_k\circ\dots\circ \gamma_0$. This idea could be made precise (as in \cite{cohen/jones/segal:1995}) using \textit{height parameterized flow lines}, which satisfy the differential equation
 \begin{equation*}\label{2eqn:ht diff EQ}
     \frac{d\gamma}{dt} = - \frac{\grad_{\gamma} f}{\|\grad_{\gamma} f\|^2}
 \end{equation*}
away from critical points. The integral curves of the associated vector field $
X$ are precisely the height-parameterized curves (see \cite[Lemma 4.8]{cohen/iga/norbury:2006}). Thus $X$ and $-\grad f$ have the same integral curves under different parametrizations. Under this reparametrization, the topology on $\overline{\mathcal M}(p,q)$ coincides with the compact-open topology as a subspace of ${\rm Map}([f(q), f(p)], M)$ (see Appendix \ref{app:top}). The advantage of reparameterizing is that we could ``glue'' flow lines together, however this approach is not strictly necessary.
\end{remark}

\begin{remark}
    Note that all the geometric constructions above depend on a prescribed Riemannian metric $g$ on $M$. We will use notation such as $\nabla_gf$, $W_g^{u}(p)$, $\M_g(p,q)$, etc. when clarification is needed.
\end{remark}

When $(f, -\grad f)$ is Morse--Smale, the compactified moduli space carries the structure of a manifold with corners (cf. \cite[Proposition 2]{cohen:2019}). This result, proved in \cite{qin:2010}, is crucial to the proof of the Cohen--Jones--Segal theorem that $B\cat C_f \cong M$ when $f$ is Morse--Smale. However, without this Morse--Smale condition, the moduli spaces may be incredibly ill-behaved, as we shall demonstrate in \cref{sec:counterex}. For the purposes of establishing a homotopy equivalence $B\cat C_f\simeq M$ for more general kinds of Morse functions $f$, we need to eliminate extremely pathological behaviors; in particular, we want the compactified moduli spaces to be locally contractible, in the following sense.

\begin{definition}
\label{def:loccon}
    A topological space $X$ is \textit{locally contractible} if for any $x\in V\subseteq X$ where $V$ is an open neighborhood of $x$, there exists a smaller open set $x\in U\subseteq V$ such that the inclusion $U\hookrightarrow V$ is null-homotopic.
\end{definition}

This condition is relatively mild in practice. For instance, manifolds and CW complexes are locally contractible. An example of a contractible but not locally contractible space is the cone on the infinite earring.

\begin{remark}
A space that satisfies \cref{def:loccon} is sometimes said to be \textit{weakly locally contractible}. Alternatively, a space $X$ is \textit{strongly locally contractible} if for each neighborhood $x\in V\subseteq X$, there is a smaller neighborhood $x\in U\subseteq V$ such that $U$ is contractible. As \cref{def:loccon} is more often used in the literature, we will take it to be our definition of locally contractible.
\end{remark}

\begin{definition}\label{def:tame}
    A Morse function $f\colon M\to \RR$ is \textit{tame} if $\overline{\M}(p,q)$ is locally contractible for all $p,q\in \crit(f)$. 
\end{definition}

We note that tameness depends on the choice of metric of $M$, although we have suppressed this in our notation.
The definition above can be seen as a loosening of the Morse--Smale condition; in particular, every Morse--Smale function is tame.

\begin{lemma}
\label{lem:cw}
For $p,q\in \crit(f)$, if $\overline{\M}(p,q)$ is locally contractible, then $\overline{\M}(p,q)$ is homotopy equivalent to a CW complex.
\end{lemma}\begin{proof}
    This claim follows immediately from the classical result that if a compact subspace $K\subseteq \R^n$ is locally contractible, then $K$ is homotopy equivalent to a CW complex (see \cite[Theorem A.7, Corollary A.8, Proposition A.11]{Hatcher}). 
\end{proof}

\begin{example}\label{ex:torus}
Let $f\colon T^2\to \RR$ be the height function on the $2$-torus. This Morse function has four critical points, as depicted below. \[
\input{figures/verticaltorus}\]
The Morse indices are $\mu(p_1) = 2$, $\mu(p_2) = \mu(p_3) = 1$, and $\mu(p_4)=0$. The uncompactified moduli spaces are\[
\M(p_1, p_2) \cong \M(p_2, p_3) \cong \M(p_3, p_4) \cong *\sqcup * \text{ and } \M(p_1, p_4) \cong (0,1)\sqcup (0,1);
\] additionally $\M(p_i, p_i) = *$ for all $i=1,2,3,4$ and all other moduli spaces are empty. This Morse function is not Morse--Smale, which can be seen for instance by noting that there are flow lines between two critical points with the same Morse index.
However, this Morse function is tame. The compactified moduli spaces are\[
\overline{\M}(p_1, p_3) \cong \overline{\M}(p_2, p_4) \cong * \sqcup * \sqcup * \sqcup * \text{ and } \overline{\M}(p_1, p_4) \cong [0,1]\sqcup [0,1]\sqcup * \sqcup *\sqcup *\sqcup *;
\] for all other pairs $p_i, p_j$, $\overline{\M}(p_i, p_j) = \M(p_i, p_j)$, which are locally contractible spaces.
\end{example}

\section{The Cohen--Jones--Segal Theorem}\label{sec:thm}

In this section, we prove the following version of the result of \cite{cohen/jones/segal:1995}.

\begin{theorem}\label{thm:loc}
Let $(M,g)$ be a smooth, closed, finite-dimensional Riemannian manifold and let $f\colon M\to \RR$ be a tame Morse function. Then there is a homotopy equivalence $B\cat C_f\simeq M$.
\end{theorem}

We begin by recalling the definition of the flow category in \cref{subsec:fc}. We then introduce some intermediary constructions used in \cite{cohen/jones/segal:1995}, before proving the main theorem in \cref{sec:pf}. Our broad strategy of proof follows the original strategy of Cohen--Jones--Segal,  by producing a continuous functor $\Theta\colon \tilde{\cat C}_f \to \cat M$ (where $ \tilde{\cat C}_f $ and $ \cat M$ are variants on the flow category $\cat C_f$ and the manifold $M$, respectively), but we use a very different approach to show that $B\Theta$ is an equivalence. In particular, we deduce the desired result from a theorem of Smale \cite{Sm:main}, which requires a point-set topological analysis of the map $B\Theta$ and the spaces involved. 

\subsection{The flow category}\label{subsec:fc}
We will construct the flow category as a \textit{topological category} in the sense of \cite{segal:1968}, i.e.\ a category internal to $\Top$, a convenient category of topological spaces; we point the reader towards \cref{sec:topcats} or \cite{segal:1968} for more details on topological categories. The flow category $\cat C_f$ of a Morse function $f\colon M\to \RR$ is the category whose objects are the critical points of $f$ and whose morphisms are the broken flow lines between them. 

\begin{definition}\label{2defn:Cf}
The \textit{flow category} of $f$ is the category $\cat C_f$ enriched in $\Top$ whose objects are the critical points of $f$ and whose homspaces are\[
\cat C_f(p,q) = \overline{\mathcal{M}}(p,q).
\] There is a natural associative composition law\[
\overline{\mathcal{M}}(p,r)\times \overline{\mathcal{M}}(r,q) \to \overline{\mathcal{M}}(p,q)
\] given by the concatenation of tuples. Given broken flow lines $\gamma_1\in \overline{\mathcal M}(p,r)$ and $\gamma_2\in \overline{\mathcal M}(r,q)$, we denote their composition by $\gamma_1\circ\gamma_2\in \overline{\mathcal M}(p,q)$. The identity map sends $p\in \crit(f)$ to the constant curve at $p$, which we denote $\id_p$, and the source and target maps send $\gamma$ to $s(\gamma)$ and $t(\gamma)$, respectively.

To endow $\cat C_f$ with an internal structure, we consider the objects, $ \ob\cat C_f=\crit(f)$, as a subspace of $M$. However, since critical points of a Morse function are isolated, this subspace has the discrete topology. The space of morphisms $\hom\cat C_f$ is topologized as the disjoint union of all the homspaces $\overline{\mathcal{M}}(p,q)$, for all pairs of critical points $p,q\in \crit(f)$. It is straightforward to check that the structure maps are all continuous, and hence $\cat C_f$ is a topological category.
\end{definition}

\begin{remark}\label{rmk:fc enriched vs internal}
We may think of of $\cat C_f$ as a category internal to $\Top$, however morally it is just a topologically enriched category, viewed as an internal category with discrete object space. For more general functions (e.g. Morse--Bott functions), the flow category will be genuinely internal.
\end{remark}

Recall that the \textit{nerve} of a topological category is the simplicial space whose space of $n$-simplices consists of $n$-composable morphisms. For the flow category, we have\[
N_n\cat C_f = \coprod_{\bf p} \cat C_f(p_0,p_1)\times \dots \times \cat C_f(p_{n-1}, p_n)
\] where the coproduct is taken over ordered length-$n$ sequences ${\bf p}=\{p_0,p_1,\dots, p_{n-1},p_n\}$ of critical points. The face maps are given by composition and the degeneracies insert identities. The \textit{classifying space} is the geometric realization of the nerve, \[B\cat C_f:=\abs{N_*\cat C_f}.\]

\subsection{Some intermediary constructions}\label{sec:intermediary}
To relate $\cat C_f$ to $M$, we follow \cite{cohen/jones/segal:1995} and introduce some intermediary topological categories.

\begin{definition}
    Let $\cat M$ be the topological category whose object space is $M$ and morphism space is also $M$. All the structure maps (source, target, identity, and composition as in \cref{def:topcat}) are the identity on $M$.
\end{definition}

The classifying space of this category is homeomorphic to $M$, since $N_*\cat M$ is the constant simplicial space on $M$. By a general result, stated in \cref{cor:BtwC}, we can replace $B\cat C_f$ with the classifying space of the twisted arrow category $\tw \cat C_f$. For the reader's convenience, we briefly recall the definition of the twisted arrow category of $\cat C_f$; see \cref{defn:tw C} for the general definition. 

\begin{definition}
    The \textit{twisted arrow category} of $\cat C_f$ is the topological category $\tw\cat C_f$ whose object space is $\hom\cat C_f = \coprod_{p,q\in \crit(f)} \overline{\M}(p,q)$. If $\gamma$ is a broken flow line from $p$ to $q$, we will sometimes write $\gamma\colon p\to q$. A morphism from $\gamma$ to $\gamma'$ in $\tw\cat C_f$ is given by $\alpha,\beta\in \hom(\cat C_f)$ such that $\gamma = \beta\circ \gamma'\circ \alpha$; such a morphism may be depicted as a commutative square\[
    \begin{tikzcd}
    p\arrow[d, swap, "{\gamma}"]  \arrow[rr, "\alpha"] &&p'\arrow[d, "{\gamma'}"]\\
    q && q'\arrow[ll, "\beta"]
    \end{tikzcd}
\] and composition is given by composing such squares in $\cat C_f$. We identify this morphism with the triple $(\alpha, \gamma', \beta) \in \hom(\cat C_f)^{\times 3}$, and the collection of morphisms is topologized as a subspace of $\hom(\cat C_f)^{\times 3}$.
\end{definition}

Applying \cref{cor:BtwC}, we then have the following result.

\begin{proposition}\label{prop:BtwCf is BCf}
    There is a homeomorphism $B\cat C_f\cong B\tw\cat C_f$. 
\end{proposition}

We now introduce a ``marked'' version of $\tw\cat C_f$, which we will then relate to $\cat M$.

\begin{definition}
    Let $\tilde{\cat C}_f$ be the topological category whose objects are pairs $(\gamma, x)$ where $\gamma\in \hom\cat C_f$ and $x\in \im\gamma$. As before, we sometimes write $(\gamma, x)\colon p\to q$. A morphism from $(\gamma, x)$ to $(\gamma', x')$ in $\tilde{\cat C}_f$ is given by $\alpha, \beta\in \hom\cat C_f$ such that $\gamma=\beta\circ\gamma'\circ\alpha$ and $x=x'$; such a morphism can be depicted as a commutative square\[
    \begin{tikzcd}
    p\arrow[d, swap, "{(\gamma,x)}"]  \arrow[rr, "\alpha"] &&p'\arrow[d, "{(\gamma',x')}"]\\
    q && q'\arrow[ll, "\beta"]
    \end{tikzcd}
\] and we denote it by $(\alpha, \beta)_x$. The topology on $\tilde{\cat C_f}$ is induced by viewing it as a subcategory of $\tw\cat C_f\times \cat M$.
\end{definition}

We may have $p=q$, in which case $\gamma$ must be the identity on $p$ and the point $x$ must also be $p$. Note that if a morphism exists it is unique (although it may be factored as several different compositions), and a morphism exists if and only if $\im \gamma'\subseteq \im \gamma$ in $M$. There is a forgetful functor $U\colon \tilde{\cat C}_f\to \tw\cat C_f$ which forgets the marked points.

\begin{lemma}\label{lem:forgetful}
    The forgetful functor $U\colon \tilde{\cat C}_f\to \tw\cat C_f$ induces a homotopy equivalence on classifying spaces. Consequently, by \cref{prop:BtwCf is BCf}, there is an equivalence $B\tilde{\cat C_f}\simeq B\cat C_f$. 
\end{lemma}\begin{proof}
    We will apply Quillen Theorem A for topological categories \cite[Theorem A]{roberts:2022}. To apply this result, we need to check the topological categories involved are well-pointed, for which we use \cref{prop:well ptd}. That is, it suffices to observe that \[
    \hom\tilde{\cat C}_f \cong \coprod_{p\succeq p' \succeq q' \succeq q} \cat C_f(p,p') \times \cat C_f(p', q') \times \cat C_f(q',q)\times [f(q'), f(p')]
    \] and $\ob \tilde{\cat C}_f\hookrightarrow \hom\tilde{\cat C}_f$ is the canonical inclusion of \[\coprod_{p\succeq q} \cat C_f(p,p) \times \cat C_f(p,q) \times \cat C_f(q,q) \times [f(q), f(p)]\] into this coproduct.
    The same observation holds for $\tw\cat C_f$ (without the interval $[f(q'), f(p')]$). Hence both topological categories are well-pointed. Moreover, by the topology on the objects of the twisted arrow category, it suffices to show that $(\cat C_f(p,q)\downarrow U)\to \cat C_f(p,q)$ is shrinkable (i.e. fiberwise homotopic to the identity) for each pair $p,q\in \crit(f)$. 
    
    Observe that the objects of this fiber category are morphisms $\gamma \to U(\gamma',x)$ in $\tw\cat C_f$, with $\gamma\in \cat C_f(p,q)$, which means $\gamma=\alpha\circ \gamma' \circ \beta$ for some $\alpha,\beta\in \hom\cat C_f$. In particular, $x\in \im(\gamma)$ and so this factors through $\gamma = U(\gamma, x)\xrightarrow{U(\alpha, \beta)_x} U(\gamma', x)$. This assignment defines a continuous natural transformation from the fiber category to the subcategory on objects of the form $\gamma \xrightarrow{=} U(\gamma, x)$ for $x\in \im(\gamma)$. The realization of this subcategory is homeomorphic to $\cat C_f(p,q)\times [f(q), f(p)]$ which is evidently fiberwise homotopic to $\cat C_f(p,q)$. 
    
    Now, by \cref{prop:BtwCf is BCf}, there is a homeomorphism $B\tw\cat C_f\cong B\cat C_f$, and consequently a homotopy equivalence $B\tilde{\cat C}_f\simeq B\tw\cat C_f\cong B\cat C_f$ as claimed.
\end{proof}

\begin{definition}
    Let $\Theta\colon \tilde{\cat C}_f\to \cat M$ be the projection $(\gamma, p)\mapsto p$ and $(\alpha, \beta)_p\mapsto \id_p$. 
\end{definition}

\begin{remark}\label{rmk:CJS gap}
In \cite{cohen/jones/segal:1995}, Cohen, Jones, and Segal define another functor $\Gamma\colon \cat M\to \tilde{\cat C}_f$ and show there are natural transformations between their compositions and the relevant identity functors, which would then realize to homotopy equivalences on the classifying spaces. Explicitly, $\Gamma\colon \cat M\to \tilde{\cat C}_f$ maps $x\mapsto (\gamma_x, x)$ and $\id_x\mapsto \id_{(\gamma_x, x)}$. While it is straightforward to check that $\Theta$ is indeed continuous, unfortunately $\Gamma$ is not. This is because the assignment $x\mapsto \gamma_x$ is not continuous in general. For instance, if some $\cat C_f(p,q)$ (which is open in $\hom\cat C_f$) contains unbroken flow lines, then its preimage under this assignment map is $W(p,q)$, which is not necessarily open in $M$.
\end{remark}

Indeed, as we will show in \cref{counter}, the claim that $B\cat C_f\simeq M$ is not true in general. However, as long as the morphism spaces of $\cat C_f$ are sufficiently well-behaved (in particular locally contractible), then $B\Theta$ is a homotopy equivalence. We now proceed with the proof of this claim.


\subsection{Proof of \cref{thm:loc}}\label{sec:pf}
We will deduce the result from the following theorem of Smale.

\begin{theorem}[\cite{Sm:main}]\label{thm:smaleloc}
Let $f\colon X\to Y$ be a proper, surjective continuous map where $X$ and $Y$ are path-connected, locally compact, separable metric spaces. Suppose $X$ is locally contractible, and for each $y\in Y$, $f^{-1}(y)$ is contractible and locally contractible. Then $f$ is a weak homotopy equivalence.
\end{theorem}

The content of our proof is verifying that this theorem applies to $B\Theta\colon B\tilde{\cat C_f}\to B\cat M$ under the given assumptions on $f$. Without loss of generality, we may assume $M$ is connected; we first verify that $B\tilde{\cat C}_f$, or equivalently $B\cat C_f$, is path-connected.

\begin{lemma}
\label{lem:connected}
    If $M$ is connected, then $B\cat C_f$ is path-connected.
\end{lemma}
\begin{proof}
It suffices to show that any two objects in $\cat C_f$ can be connected by a zig-zag of morphisms. Suppose for contradiction that this is not the case, so there exists some critical points $p_0$ and $q_0$ that cannot be connected by any zig-zag of gradient flow lines. Write $\crit(f) = C_1\sqcup C_2$ where $C_1 := \{p\in \crit(f)\mid [p] = [p_0] \in \pi_0(\cat C_f)\}$ and $C_2 = \crit(f)\setminus C_1$. Then, by assumption, $C_1,C_2\neq \varnothing$ (since in particular $q_0\in C_2$) and $\M(p,q)=\varnothing = \M(q,p)$ for any $p\in C_1$, $q\in C_2$. 


Let \[M_1=\{x\in M\mid x\in W(p,p') \,\;\text{for}\;\, p,p'\in C_1\}\text{ and }M_2=\{x\in M\mid x\in W(q,q') \,\;\text{for}\;\, q,q'\in C_2\},\] so then $M=M_1\cup M_2$ and $M_1\cap M_2=\varnothing$.
We now show that $M_1$ is closed. Suppose $x_n\to x$ for $x_n\in M_1$ with $x_n\in W(p_n,q_n)$  for $p_n, q_n\in C_1$, $n\geq 1$. Since there are only finitely many critical points, there is one pair $(p_l,q_l)$ that occurs infinitely often, so by extracting a subsequence we assume that $x_n\in W(p_l,q_l)$ for all $n$. Let $\gamma_{n}\in \M(p_l,q_l)$ be the unbroken flow line containing $x_n$; by compactness of $\overline{\M}(p_l,q_l)$, there is a subsequence $\gamma_{{n_i}}\to \gamma\in \overline{\M}(p_l,q_l)$ with $x_{n_i}\in \im(\gamma_{n_i})\subseteq M_1$. This implies $x\in \im\gamma\subseteq M_1$; one way to see this is to use the distance function of \cref{def:distance fn}: letting $d(x_{n_i}, \gamma):= \inf_{y\in \gamma} d_M(x_{n_i}, y)$, we have\[
d(x_{n_i},\gamma)\leq d(\gamma_{n_i},\gamma) \text{ for all } n_i
\] so taking $n_i\to \infty$ we have $d(x,\gamma)\leq d(\gamma,\gamma)=0$, which means $x$ lies in the image of $\gamma$.
Thus $M_1$ is closed, and a similar argument shows $M_2$ is closed, which contradicts the fact that $M$ is connected. Hence $B\cat{C}_f$ is path-connected.
\end{proof}

Since $M$ is assumed to be a closed manifold, it automatically meets the conditions of Smale's theorem (\cref{thm:smaleloc}) and so we just need to check that $B\tilde{\cat C}_f$ satisfies the desired point-set topological properties.

\begin{lemma}\label{lem:BCf nice}
    The classifying space $B\tilde{\cat C}_f$ is compact Hausdorff, second countable and locally contractible. Hence it is a locally compact, separable metric space.
\end{lemma}
\begin{proof}
Note that for any fixed broken flow line $\gamma\in \overline{\M}(p,q)$, a point $x\in \im\gamma$ is uniquely determined by $f(x)\in [f(q),f(p)]$, and therefore we have
\[  N_0\tilde{\cat C}_f=\ob\tilde{\cat C}_f\cong \bigsqcup_{p,q\in \crit(f)} \overline{\M}(p,q)\times [f(q), f(p)].  \]
Similarly, it is straightforward to verify that there is a homeomorphism
\begin{equation}
\label{eqn:nerve}
    N_n\tilde{\cat C}_f\cong\bigsqcup_{p_i\in \crit(f)}\overline{\M}(p_0,p_1)\times\cdots\times \overline{\M}(p_{2n},p_{2n+1})\times [f(p_{n+1}), f(p_{n})].
\end{equation} 
Note that the degenerate simplices of $N_n\tilde{\cat  C}_f$ are those for which $p_i=p_{i+1}$ and $p_{2n-i} = p_{2n+1-i}$ for some $0\leq i\leq 2n$ ($i\neq n$), from which it follows that $N_*\tilde{\cat C}_f$ is a split simplicial space in the sense of \cref{def:splitsim}. This observation will allow us to apply lemmas from \cref{sec:topcats} to deduce the desirable topological properties of $B\tilde{\cat C}_f$ from the analogous properties of $N_n\tilde{\cat C}_f$. 
In particular, since each $\overline{\M}(p,q)$ is compact Hausdorff, second countable and locally contractible, by \cref{eqn:nerve} we see that $N_n\tilde{\cat C}_f$ is compact Hausdorff, second countable and locally contractible for every $n\leq N$.

Note that $\tilde{\cat C_f}$ is acyclic (meaning there are no non-trivial cycles of morphisms) since there is a morphism $(\gamma, x)\to (\gamma', x')$ only if $x=x'$ and $\im(\gamma') \subseteq \im(\gamma)$.
Together with the finiteness of $\crit(f)$, this implies that all nondegenerate simplices in $N_*\tilde{\cat C}_f$ have dimension at most $N:= |\crit(f)|$. Hence $B\tilde{\cat C}_f\cong \abs{{\rm sk}_N N\tilde{\cat C}_f}$, where
\[ {\rm sk}_N(N\tilde{\cat C}_f)=\Big(\bigsqcup_{0\leq n\leq N} N_n\tilde{\cat C}_f\times \Delta^n\Big)/\sim \] is the $N$-skeleton of $B\tilde{\cat C}_f$. 
Therefore $B\tilde{\cat C}_f$ is the geometric realization of a finite-dimensional, split simplicial space which is levelwise compact Hausdorff, second countable, and locally contractible. Hence, by \cref{lem:simloc,lem:sep}, $B\tilde{\cat C}_f$ is itself is compact Hausdorff, second countable and locally contractible, which implies it is separable, metrizable and locally compact.
\end{proof}

Finally, we verify the desired properties of the fibers of  $B\Theta\colon B\tilde{\cat C}_f\to M$. For $x\in M$, let $\Theta^{-1}(x)\subseteq \tilde{\cat C_f}$ denote the full subcategory on objects of the form $(\gamma, x)$ for $\gamma\in \hom\cat C_f$. 

\begin{lemma}
\label{lem:fiber}
For any $x\in M$, there is a homeomorphism $(B\Theta)^{-1}(x)\cong B(\Theta^{-1}(x))$. Moreover, $(B\Theta)^{-1}(x)$ is contractible and locally contractible.
\end{lemma}
\begin{proof}
By definition of geometric realization, we see that $(B\Theta)^{-1}(x)\subseteq B\tilde{\cat C_f}$ is the subspace consisting of points $[((\gamma_0, x),\cdots,(\gamma_n, x);w)]$ for $((\gamma_0, x),\cdots,(\gamma_n, x))\in N_n\tilde{\cat C}_f$ and $w\in \Delta^n$, $n\geq 0$. Since the face and degeneracy maps in the nerve preserve the point $x$, we see that this subspace is precisely
\[ B(\Theta^{-1}(x))=\Big(\bigsqcup_{n\geq 0}N_n \Theta^{-1}(x) \times \Delta^n\Big)/\sim.\]
We may therefore write $B\Theta^{-1}(x)$ unambiguously.

Note that $\Theta$ is surjective (so in particular $\Theta^{-1}(x)$ is non-empty) since for any $x\in M$, we have $(\gamma_x, x)\in \Theta^{-1}(x)$ where $\gamma_x$ is the minimal unbroken flow line through $x$ as in \cref{def:minimal flow} (noting that $\gamma_x=\id_x$ for $x\in \crit(f)$). In fact, this object is terminal, since for any $\gamma\in \ob\cat C_f$ such that $x\in \im\gamma$, we have $\im\gamma\supset \im\gamma_x$ and so there is a unique morphism $(\gamma,x)\to (\gamma_x,x)$ in $\tilde{\cat C}_f$. Since $\Theta^{-1}(x)$ has a terminal object, $B\Theta^{-1}(x)$ is contractible (\cref{3rmk:initial gives contractible cl sp}).

Finally, to see that this fiber is locally contractible, we observe that
\[ N_n\Theta^{-1}(x) =\bigsqcup_{p_i,q_i\in \crit(f)}\overline{\M}(p_0,p_1)\times \cdots \overline{\M}(p_n,p)\times\overline{\M}(q,q_n)\times \cdots\overline{\M}(q_{1},q_0) \]
is locally contractible for all $n\geq 0$, where $p,q\in \crit(f)$ are the endpoints of the minimal unbroken flow line $\gamma_x$. Since the non-degenerate simplices of $N_n\Theta^{-1}(x)$ consist of those factors for which $p_i\neq p_{i+1}$ or $q_i\neq q_{i+1}$ for every $0\leq i\leq n$ (taking $p_{n+1}=p$ and $q_{n+1}=q$), we see that  $N_*\Theta^{-1}(x)$ is a split simplicial space (in the sense of \cref{def:splitsim}), and thus by \cref{lem:simloc}, the fiber $B\Theta^{-1}(x)$ is locally contractible for any $x\in M$.
\end{proof}

\begin{theorem}
    The continuous functor $\Theta\colon \tilde{\cat C}_f\to \cat M$ induces a homotopy equivalence on classifying spaces. Consequently, \cref{thm:loc} holds.
\end{theorem}\begin{proof}
We have shown that the spaces $B\tilde{\cat C}_f$ and $M$ are path-connected (\cref{lem:connected}), compact Hausdorff, separable, metrizable, and locally compact (\cref{lem:BCf nice}), and hence $B\Theta$ is proper. We have also verified the necessary condition on the fibers, and so we may conclude that $B\Theta$ is a weak homotopy equivalence by \cref{thm:smaleloc}. 

Combined with \cref{lem:forgetful}, we therefore have shown that there is a weak equivalence $B\cat C_f\simeq M$. Finally, by \cref{lem:cw}, the morphism spaces of $\cat C_f$ are homotopy equivalent to CW complexes, so by \cref{lem:scw} $B\cat C_f$ is also homotopy equivalent to a CW complex. By Whitehead's theorem, this implies that there is a homotopy equivalence $B\cat C_f\simeq M$.
\end{proof}


\section{A counterexample to the general claim}
\label{counter}

In this section, we will produce a counterexample to the general statement of the original claim made by Cohen--Jones--Segal in \cite{cohen/jones/segal:1995}; that is, we will produce a Riemannian manifold $(M,g)$ and Morse function $f\colon M\to \RR$ so that $B\cat C_f\not\simeq M$. By \cref{thm:loc}, we know that such a counterexample must have ``topologically ill" Morse trajectory spaces, i.e. we must produce moduli spaces of broken flows which are not locally contractible. We will construct a moduli space whose compactification is $Z=\{0\}\cup\{1/n\mid n\geq1\}$.

The manifold underlying our counterexample is $M=S^2\times S^1$. The idea is to start with a standard embedding of $S^2\times S^1$ into $\RR^4$ and then to perturb the metric slightly so that one of the morphism spaces becomes $Z$. The general strategy for such a perturbation is outlined in \cref{sec:twisting} and the specific counterexample is analyzed in \cref{sec:counterex}.

\subsection{A perturbation on the Morse flows}\label{sec:twisting} 

The goal of this subsection is to show that any compact subset $Z\subseteq\RR$ can arise as a moduli space of broken flows. This moduli space will arise from a ``perturbation operation'' we can perform on a $3$-manifold and a Morse function satisfying some conditions, detailed as follows.

\begin{assumption}
Let $(M,g_0)$ be a closed oriented Riemannian $3$-manifold, and suppose that $f\colon M\to \RR$ is a Morse function with two critical points $p,q\in \crit(f)$ such that
\begin{itemize}
\item[(i)] $\mu(p)=2$, $\mu(q)=1$.
\item[(ii)] $W^u(p)\setminus \{p\}=W^s(q)\setminus \{q\}\cong S^1\times \mathbb{R}$.
\end{itemize}
Note that (ii) implies $\mathcal{M}(p,q)=(W^u(p)\setminus \{p\})/\mathbb{R}=(W^s(q)\setminus\{q\})/\mathbb{R}\cong S^1$.
\label{assum}
\end{assumption}

In other words, we want a Morse function on a manifold which has two adjacent critical points $p,q$ such that all gradient flow lines that start at $p$ end at $q$ (see Figure \ref{fig:u=s}). Our counterexample in \cref{sec:counterex} will begin from such a set up with $M=S^2\times S^1$.

\begin{figure}[h]
    \centering
    \input{figures/coincide}
    \caption{\cref{assum}}
    \label{fig:u=s}
\end{figure}

\begin{remark}
One way to construct such a manifold is to start with a smooth manifold with boundary $M_1$ and a non-negative Morse function $f\in C^{\infty}(M_1)$, such that $f\equiv 0$ along the boundary. Let $p\in \crit(f)$ be the critical point with the smallest critical value and $\mu(p)=2$; then let $M$ be the double of $M_1$ and extend $f$ symmetrically on $M$, in other words in the double define the function to be $-f$, and let $q$ be the corresponding critical point of $p$ in the double. Then the unstable manifold of $p$ and the stable manifold of $q$ will naturally coincide because of the double construction, as is illustrated in Figure 1.
\end{remark}

Under Assumption \ref{assum}, we can carry out a perturbation on the Morse flows to make the trajectory space $\mathcal{M}(p,q)$ as pathological as we would like. 

\begin{theorem}
    Let $(M,g_0)$ be a closed Riemannian $3$-manifold equipped with a Morse function $f$ satisfying Assumption \ref{assum}. Given any compact subset $Z\subseteq \R$, there exists a metric $g$ on $M$ such that, with respect to the new metric, $\M_{g}(p,q)\cong Z$ for those critical points $p$ and $q$ guaranteed by Assumption \ref{assum}.
    \label{thm:twisting}
\end{theorem}

As we will see in the proof, the new metric $g$ only differs from $g_0$ outside of a small neighborhood of some certain level set between $p$ and $q$, and can be made arbitrarily close to $g_0$. Our proof will make use of the following classical result, e.g. mentioned in \cite{smale:1961}.

\begin{lemma}
\label{smalegradient}
    Let $(M,g_0)$ be a closed Riemannian manifold, $f\colon M\to \RR$ be a Morse function, and $X$ be a $C^{\infty}$ vector field on $M$ which satisfies the following conditions:
\begin{enumerate}
    \item [(1)] The critical points of $f$ are exactly the singular points of $X$, i.e. $p\in M$ with $X_p=0$.
    \item [(2)] For each $p\in \crit(f)$, there is a neighborhood $U_p$ of $p$ such that $X=-\nabla_{g_0} f$ on $U_p$. Additionally, for $p\neq p'$, the neighborhoods $U_p$ and $U_{p^{\prime}}$ are disjoint.
    \item [(3)] For each $q\in M$ such that $X_q\neq 0$, $X_q$ is transverse to the level hyperplane of $f$ at $q$, which is defined locally by $\{ x\in M \mid f(x)=f(q)\}$ (and is a submanifold because $q$ is not a critical point).
\end{enumerate}
Then there is a Riemannian metric $g$ on $M$ such that $X=-\nabla_g f$, and $g=g_0$ over the locus where $X=-\nabla_{g_0} f$.
\end{lemma}

\begin{proof}
For each $q\notin \crit(f)$, let $L=\{ x\in M \mid f(x)=f(q) \}$ be the level hyperplane containing $q$, which is a submanifold near $q$ with $T_qL=\{v\in T_qM \mid f_*v=0\}$. Therefore, by condition (3), we have $f_*X_q\neq 0$. In fact, $f_*X_q<0$ since condition (2) says that $f_*X=-g_0(\nabla f,\nabla f)\leq 0$ in $U_p$ for $p\in \crit(f)$ (and so implies the claim for $q\not\in\crit(f)$ by connectivity). 

By condition (3), we have $T_qM=T_qL\oplus \langle X_q\rangle$. Now, given $v_1,v_2\in T_qM$, we define $g(v_1,v_2)=\lambda_1\lambda_2+g_0(w_1,w_2)$ where\[
v_i = -\lambda_i\cdot \frac{1}{\sqrt{-f_*X_q}}X_q+w_i
\] for $\lambda_i\in \RR$ and $w_i\in T_qL$,  $i=1,2$. It is straightforward to verify that $g$ is an inner product on $T_qL$. In addition, if $X_q=-(\nabla_{g_0} f)_q$, then $g=g_0$ at $q$. So we have $g=g_0$ on $U_p\setminus \{p\}$, and we naturally define $g=g_0$ at the critical points. Hence $g$ defines a Riemannian metric on $M$ which agrees with $g_0$ at the points where $X=-\nabla_{g_0}f$. Finally, we observe that $X=-\nabla_g f$, since \[
g(X_q,v)= -\sqrt{-f_*X_q}\cdot f_*v/\sqrt{-f_*X_q}=-f_*v
\]
using that $\lambda_i=f_*v_i/\sqrt{-f_*X_q}$.
\end{proof}

We can thus carry out our perturbation on the gradient vector field instead of directly on the Riemannian metric.

\begin{proof}[Proof of Theorem \ref{thm:twisting}]
    Let $X=-\nabla_{g_0} f$ be the negative gradient vector field of $f$ with respect to $g_0$ and let $a\in (f(q),f(p))$ be a regular value of $f$. Then $L:= \{x\in M\mid f(x)=a\}$ where $a\in (f(q),f(p))$ is a level hyperplane lying between $p$ and $q$, and we know that \[\M(p,q)\cong W^u(p)\cap L=W^s(q)\cap L\cong S^1\] by \cref{assum}. We denote this submanifold by $N:=W^u(p)\cap L=W^s(q)\cap L$.

    By standard Morse theoretic arguments, there is a neighborhood of $L$ in $M$ that is diffeomorphic to $L\times[-1,1]$, where $L$ is identified with $L\times\{0\}$. By the tubular neighborhood theorem, $N$ has a neighborhood $N'$ in $L$ such that $(N^{\prime},N)\cong (N\times [-1,1], N\times\{0\})$. 
    
    We will construct a vector field $Y$ on $M$ which is identical to $X$ outside the ``box" $N^{\prime}\times [0,1]\subseteq L\times [0,1]\subseteq M$ (as shown in \cref{fig:mln}). We denote the coordinates in the box $N^{\prime}\times [0,1]\cong N\times [-1,1]\times [0,1]$ by $(x,y,z)$ for $x\in N,y\in[-1,1],$ and $z\in[0,1]$.

    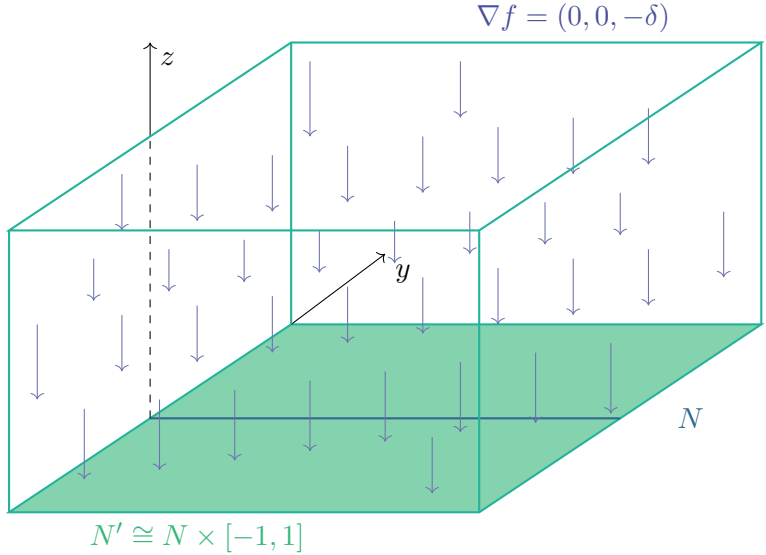
\begin{figure}[h]
        \centering
        \input{figures/box}
        \caption{The box $N\times[-1,1]\times[0,1]\subseteq M$}
        \label{fig:mln}
    \end{figure}
    
    The vector field $X=-\nabla f$ restricts to $(0,-\delta(w,z))\in T(L\times [-1,1])=TL\times \R$ at $(w,z)\in L\times [-1,1]$ for some smooth positive function $\delta\colon L\times[-1,1]\to \R^+$.
    In particular, on $N^{\prime}\times [0,1]$, we have $X=(0,0,-\delta)$. We define the vector field $Y$ on $N^{\prime}\times [0,1]$ by
    \[ Y(x,y,z)=(0,\theta(x)\Psi(y,z),-\delta(x,y,z)), \]
    where 
    \begin{itemize}
        \item $\Psi\colon [-1,1]\times [0,1]\to \R$ is a smooth function which is non-negative, satisfies the inequality $\Psi(0,1/2)>1$, and vanishes near the boundary.
        \item $\theta\colon N\cong S^1=\R\cup\{\infty\}\to\R$ is a smooth function with the zero set $\theta^{-1}(0)=Z$. The fact that such a function $\theta$ exists is classical, see \cite{whitney}, although a possible construction of $\theta$ is given in \cref{rmk:construct theta}.
    \end{itemize}
    Observe that $Y=(0,0,-\delta)=X$ near the boundary of $N^{\prime}\times[0,1]\subseteq M$, so we may define $Y$ to agree with $X$ on $M\setminus N^{\prime}\times[0,1]$ to get a smooth vector field on $M$. Note that $Y$ satisfies the conditions in \cref{smalegradient}. In fact, $Y=X=-\nabla f$ outside the ``box" $N^\prime \times [0,1]$ and inside the box $Y$ is transversal to the level hyperplanes defined by $z=z_0$ for some $z_0\in [0,1]$. Now applying  \cref{smalegradient}, we conclude there is a metric $g$ such that $Y= -\nabla_{g} f$, noting that (by the proof of the lemma) this metric differs from $g_0$ only possibly inside the box $N^{\prime}\times [0,1]$.

    We now show that $\M_g(p,q)\cong Z$ with respect to the new metric $g$. Since $g=g_0$ outside $N^{\prime}\times [0,1]$, we have 
    \[ W^u_g(p)\cap (L\times \{1\})=N\times\{0\}\times\{1\} \text{ and } \quad W^s_g(q)\cap (L\times \{0\})=N\times\{0\}\times\{0\}. \]
    Let $\gamma$ denote the gradient flow line that starts at $(x_0, 0,1)\in \{N\}\times \{0\}\times \{1\}$ and write \[\gamma(t)=(x(t),y(t),z(t))\in N\times[-1,1]\times[0,1].\] By definition, we have $\frac{d}{dt}\gamma=-\nabla_{g}f=Y(\gamma)$, so
    \[ \frac{d}{dt}x(t)=0, \quad \frac{d}{dt}y(t)=\theta(x(t))\Psi(y(t),z(t)), \quad \frac{d}{dt}z(t)=-\delta(x(t),y(t),z(t)). \]
    Therefore, $x(t)\equiv x_0$ and so $\gamma$ lies in the ``slice'' $\{x_0\}\times[-1,1]\times[0,1]$; see \cref{fig:slice}. 

    \begin{figure}[h!]
        \centering
        \input{figures/slice}
        \caption{Perturbation of the vector field at $x=x_0$}
        \label{fig:slice}
    \end{figure}
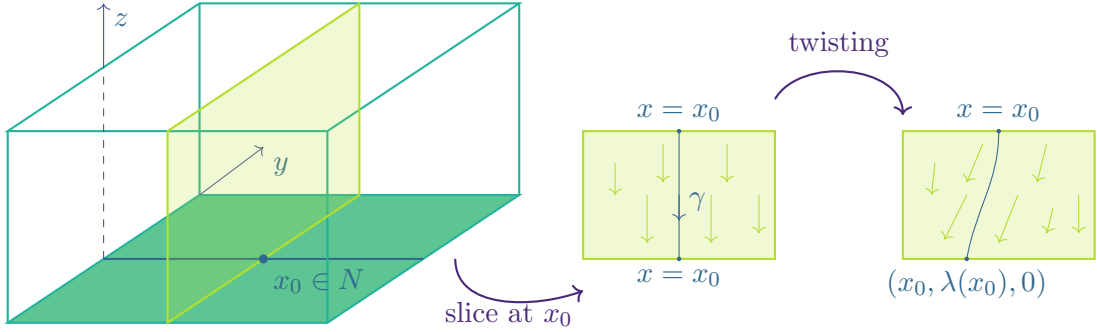

    Let $\lambda\colon N\to [-1,1]$ be the function which describes how $N$ moves under the perturbed vector field, setting $\lambda(x_0)$ to be the $y$-coordinate of the point $\im\gamma\cap (L\times \{0\})\in N'$. We claim that $\lambda(x_0)$ has the same sign as $\theta(x_0)$ for all $x_0\in N$, and in particular the same zero set. To see this, note that if $\theta(x_0)=0$, then $\frac{d}{dt}y(t)=0$ and so $\lambda(x_0)=0$ as well (since $y(t)$ starts at $y_0=0$). If $\theta(x_0)>0$, we may consider $t_0>0$ satisfying $\gamma(t_0)=(x_0,\lambda(x_0),0)$, then $\lambda(x_0) = \int_0^{t_0}\frac{dy}{dt}dt \geq 0$, i.e.\ $t_0$ is the time when $\gamma$ hits the hyperplane $z=0$. If $\lambda(x_0)=0$, then $\Psi(y(t),z(t))\equiv 0$ for all $t$, which contradicts that $\Psi(0,1/2)>0$. Hence $\lambda(x_0)>0$.
    Similarly, if $\theta(x_0)<0$, then $\lambda(x_0)<0$. 

    By the above analysis, we now have $$W^u_{g}(p)\cap (L\times \{0\})=\{(x_0,\lambda(x_0),0)\in N\times[-1,1]\times[0,1]\mid x_0\in N\},$$
    and therefore
    \[ \M_g(p,q)\cong W^u_g(p)\cap W^s_g(q)\cap (L\times \{0\})=\{ (x_0,0,0)\mid \lambda(x_0)=0 \}\cong Z. \]
    \begin{figure}[h]
        \centering
\includegraphics[width=\linewidth]{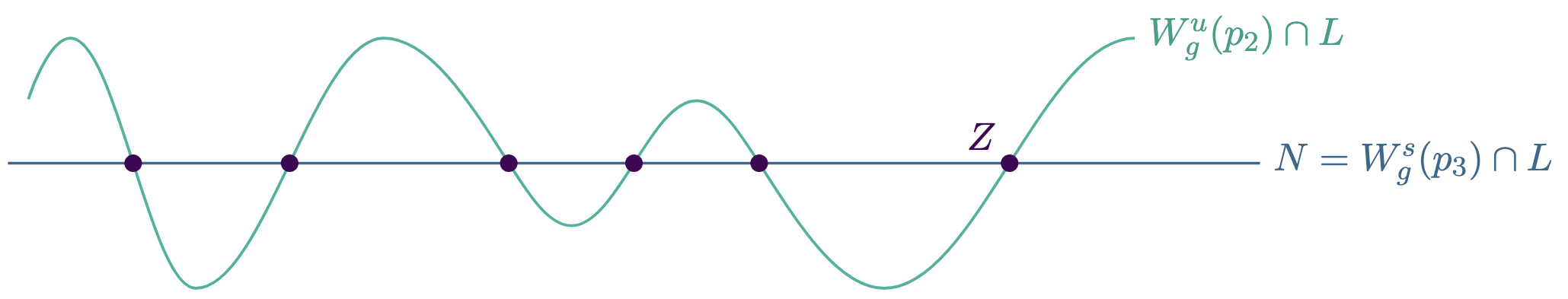}
        \caption{The perturbed moduli space}
        \label{fig:on l 1}
    \end{figure}
\end{proof}

\begin{remark}\label{rmk:construct theta}
    One way to construct the function $\theta$ is as follows: write the complement $S^1\setminus Z$ as a disjoint union of open intervals $(a_i, b_i)$. On each interval $(a_i,b_i)$, define the function by $\theta(x)=e^{\frac{1}{(x-a_i)(x-b_i)}}$ and let $\theta(x)=0$ for $x\in Z$. It is straightforward to verify that $\theta$ is smooth with the zero set $Z$.
\end{remark}

\begin{remark}
    Using a similar method as above, one can show that any compact subset $Z\subseteq \R^{n}$ can arise as a Morse trajectory space, i.e. it is possible to construct a closed Riemannian $(n+2)$-manifold $M$ along with a Morse function $f\colon M\to \RR$ with $\M(p,q)\cong Z$ for some pair of critical points $p,q\in \crit(f)$.
\end{remark}

\subsection{A counterexample}\label{sec:counterex}

In this section we give a counterexample to the general claim that $B\cat C_f\simeq M$ for an arbitrary Morse function $f$, using the twisting operation developed in \cref{sec:twisting}. 

\begin{theorem}
\label{thm:counter}
    There is a Riemannian metric $g$ on $S^2\times S^1$, as well as a Morse function $f$ on it such that $H_3(B\cat C_f)=0$, where $\cat C_f$ is the flow category for the tuple $(S^2\times S^1,g,f)$. Hence $B\cat C_f$ cannot be (weakly) homotopy equivalent to $S^2\times S^1$.
\end{theorem}

The way we construct such a metric on $S^2\times S^1$ is to first produce a ``standard" metric on $S^2\times S^1$ equipped with a height function $f$ as a Morse function, then to use Theorem \ref{thm:twisting} to construct a perturbed metric with whose Morse trajectory spaces are not locally contractible.

\begin{construction}[The standard tuple]
\label{cons:std}
    Fix any $0<r<1$ and embed $S^2\times S^1$ into $\R^4$ via
    \[ \iota\colon S^2\times S^1\to \R^4,\quad (x,\theta)\mapsto((1+r\cos\theta)x,r\sin \theta). \]
    Here we view $S^2\subseteq \R^3$ as the standard sphere; see Figure \ref{fig:s2s1} for an illustration of this embedding. Let $g_0$ be the metric on $S^2\times S^1$ induced by the standard metric on $\R^4$ and the embedding $\iota$.

    \begin{figure}[h]
        \centering
        \input{figures/S2S1}
        \caption{$S^2\times S^1\subseteq \R^4$}
        \label{fig:s2s1}
    \end{figure}
    Define a smooth function $f\colon S^2\times S^1\to \R$ by $f(x, \theta) = (1+r\cos\theta)x_1$ for a point $x=(x_1,x_2,x_3)\in S^2\subseteq \R^3$. Then $f$ is a height function in the sense that $f$ is the first coordinate under the embedding $\iota$.
\end{construction}

We now analyze the standard tuple $(S^2\times S^1,g_0,f)$.

\begin{lemma}
   The function $f$ is Morse with four critical points \begin{align*}
       p_1 &= (1,0,0,0) & \text{with }\mu(p_1)=3,\\
       p_2 &= (1,0,0,\pi) & \text{with } \mu(p_2)=2,\\
       p_3 &= (-1,0,0,\pi) &\text{with } \mu(p_3) = 1,\\
       p_4 &= (-1,0,0,0) & \text{with } \mu(p_4) = 0.
   \end{align*}
\end{lemma}
\begin{proof}
    Suppose $q=(x_1,x_2,x_3,\theta)\in S^2\times S^1$. If $x_1\neq \pm 1$, there is a small neighborhood of $q$ that is parameterized by the coordinates $(x_1,x_2,\theta)$. Hence $df/d x_1=1+r\cos \theta>0$ in this neighborhood, so $df_q\neq 0$.

    Therefore all critical points $q\in \crit(f)$ must have $x_1=\pm 1$, and are therefore of the form $q=(\pm 1, 0, 0; \theta)$. There is a parametrization of a neighborhood of such a $q$ given by
    \[(x_2,x_3,\theta)\mapsto (\pm\sqrt{1-x_2^2-x_3^2},x_2,x_3,\theta)\] and so $(df/dx_2)_q=(df/dx_3)_q=0$. Hence $df_q=0$ if and only if $(df/d\theta)_q=-r\sin \theta x_1=0$, i.e. if and only if $\theta\in\{0,\pi\}$.
    
    Therefore $f$ has $4$ critical points, given by $p_1, p_2, p_3, p_4$ as listed above. It is straightforward to verify that $f$ satisfies the Morse condition on these critical points, and that the Morse indices are $\mu(p_1)=3,\mu(p_2)=2,\mu(p_3)=1,$ and $\mu(p_4)=0$.
\end{proof}

Next we analyze the behavior of gradient flow lines. Henceforth, we view $S^2\times S^1\subseteq \R^4$ via the embedding $\iota$ and write the coordinates of $\R^4$ as $\R^4=\{(a,b,c,y)\mid a,b,c,y\in \R\}$. 

\begin{lemma}
\label{lem:stdgradient}
The (uncompactified) moduli spaces of flow lines associated to $(S^2\times S^1,g_0,f)$ may be described as follows:
    \begin{enumerate}
    \item $\M(p_1, p_2) = \{\gamma^{+}, \gamma^-\}$ with $\im(\gamma^+)\subseteq \{y>0\}$ and $\im(\gamma^-)\subseteq \{y<0\}$.
    \item $\M(p_3,p_4) = \{\sigma^{+}, \sigma^-\}$ with $\im(\sigma^+)\subseteq \{y>0\}$ and $\im(\sigma^-)\subseteq \{y<0\}$.
    \item $\M(p_2,p_3)\cong S^1$.
    \end{enumerate}
    Any other gradient flow line is in $\M(p_1, p_4)$ and all other moduli spaces are empty.
\end{lemma}

\begin{proof}

We first study $\M(p_1,p_2)$. Since $\mu(p_2)=2$, we know that there are homeomorphisms $W^s(p_2)\setminus\{p_2\}\cong \R\setminus \{0\} \cong S^{0}\times \R$; hence there are exactly two gradient flow lines pointing towards $p_2$, both of which must come from $p_1$ (since flow lines must decrease the value of the Morse function). Therefore $\M(p_1,p_2)$ consists of two points, as claimed.

To describe these two gradient flow lines more concretely, we note that the map
    \[ T\colon \R^4\to \R^4,\quad (a,b,c,y)\mapsto (a,-b,-c,y) \]
    induces an isometry on $S^2\times S^1$ and commutes with $f$. Define
    \[ 
    \begin{aligned}
        \gamma^{+}(t)&=(1+r\cos \pi t,0,0,r\sin \pi t),\\
        \gamma^{-}(t)&=(1+r\cos \pi t,0,0,-r\sin \pi t)
    \end{aligned}\]
    for all $0\leq t\leq 1$. 
    Then the images $\im(\gamma^{\pm})$ are both fixed by $T$, so $-\nabla f|_{\gamma^{\pm}}$ must also be fixed by $T$ because $T$ commutes with $f$. Thus $-\nabla f$ is tangent to the circle $\im(\gamma^+)\cup \im(\gamma^-)$. By checking the direction, we can show that $\gamma^{\pm}$ are two gradient flow lines from $p_1$ to $p_2$ (up to a positive reparameterization). Hence $\M(p_1,p_2)=\{\gamma^{\pm}\}$, where the $\pm$ sign here means $\gamma^+\subseteq \{y>0\}$ and $\gamma^-\subseteq \{y<0\}$.
    The argument for $\M(p_3,p_4)$ is the same, with $\sigma^{\pm}(t) = (-1+r\cos \pi t, 0, 0, \pm r\sin \pi t)$.

    Now we consider $\M(p_2,p_3)$. Note that the reflection
    \[ R\colon \R^4\to \R^4, \quad (a,b,c,y)\mapsto (a,b,c,-y) \]
    induces an isometry on $S^2\times S^1$ and commutes with $f$. Hence on the fixed point set of $R$, $(S^2\times S^1)\cap \{y=0\}$, the vector field $-\nabla f$ is also fixed by $R$ and hence tangent to the plane $\{ y=0\}$. Therefore, a gradient flow line either lies entirely in $\{y>0\}$, $\{ y<0 \}$, or $\{ y=0 \}$.

    Let $M_0 = (S^2\times S^1)\cap \{y=0\}$. 
    Then\[
    M_0 = (1+r)S^2 \sqcup (1-r)S^2 \subseteq \RR^3 \times \{0\}\subseteq \RR^4
    \] consists of two disjoint spheres, where $(1\pm r)S^1$ denotes the sphere of radius $1\pm r$ at the origin, and \[
    L:= M_0 \cap f^{-1}(0) = (1+r)S^1 \sqcup (1-r)S^1 \subseteq \{(0,b,c, 0)\mid b,c\in \RR\}\subseteq \RR^4
    \] gives the two equators of those spheres. Now, for each point $q\in (1-r)S^1 \subseteq L$, the gradient flow line passing through $q$ must lie entirely in $M_0$ as previously noted; hence $\gamma$ starts at $p_2$ and ends at $p_3$, since these are the only critical points which lie in the same component of $M_0$ as $q$. Hence \[
    S^1\cong (1-r)S^1 \subseteq L\cap  W^u(p_2) \text{ and }  S^1\cong (1-r)S^1 \subseteq L\cap W^s(p_3),
    \] and on the other hand $W^u(p_2)\cap L\cong S^1\cong W^s(p_3)\cap L$ (since the un/stable manifolds are homeomorphic to $2$-dimensional disks). Hence $W^u(p_2)\cap L=W^s(p_3)\cap L=(1-r)S^1\cong S^1$ and so $\M_{g_0}(p_2,p_3)\cong W^u(p_2)\cap W^s(p_3)\cap L\cong S^1$.

    The above analysis also implies that any gradient flow line ending at $p_3$ must start at $p_2$ and any gradient flow line starting at $p_2$ must end at $p_3$. Hence $\M(p_1, p_3) = \M(p_2, p_4) = \varnothing$, and any other gradient flow line not yet considered must be in $\M(p_1, p_4)$. We emphasize that any $\gamma\in \M(p_1,p_4)$ either lies entirely in $\{y\geq 0\}$ or $\{ y\leq 0 \}$.
\end{proof}

\begin{figure}[h!]
    \centering
    \input{figures/y=0}
    \caption{$M_0=(S^2\times S^1\cap\{y=0\})$}
    \label{fig:y=0}
\end{figure}
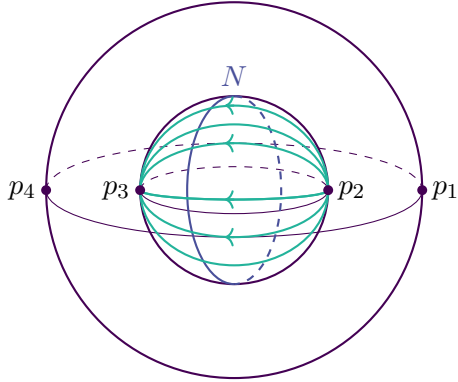

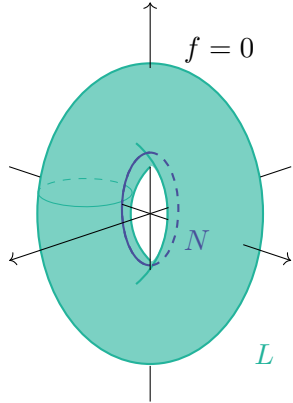
\begin{figure}[h!]
    \centering
    \input{figures/f=0}
    \caption{$L=(S^2\times S^1)\cap\{f=0\}$}
    \label{fig:f=0}
\end{figure}

The point of \cref{lem:stdgradient} is to verify the conditions of \cref{assum}. Now we apply \cref{thm:twisting} for the specific compact subset $$Z=\left\{1/n \mid n\geq 1\right\}\cup \{0\}\subseteq \R,$$ which is not locally contractible. Recall that in the proof of \cref{thm:twisting}, the perturbation depended on a choice of smooth function $\theta\colon S^1\to \R$ with zero set $\theta^{-1}(0)=Z$. To aid our analysis of the compactified moduli spaces, we will make a specific choice of function $\theta$.

\begin{construction}[An oscillating function]    \label{cons:oscillating}
    Write $S^1=\R\cup\{\infty\}$ and 
    \[ Z=S^1\setminus\bigsqcup_{k\geq0} I_k \]
    where $I_k=(\frac{1}{k+1},\frac{1}{k})$ for $k\geq 1$ and $I_0=(-\infty,0)\cup (1,\infty)\cup \{\infty\}$.
    Define $I^+=\bigsqcup\limits_{k\geq 0}I_{2k}$ and $I^-=\bigsqcup\limits_{k\geq 0}I_{2k+1}$, and note that the closures are $\overline{I^+}=I^+\cup Z$ and $\overline{I^-}=I^-\cup Z$. Define a smooth function $\theta$ by
    \[ \theta(x)=\left\{\begin{aligned}
        &e^{\frac{1}{(x-a_i)(x-b_i)}}, &x\in (a_i,b_i)\subseteq I^+,\\
        &e^{-\frac{1}{x(x-1)}},  &x\in I_0\subseteq I^+, \\
        &-e^{\frac{1}{(x-a_i)(x-b_i)}},  &x\in (a_i,b_i)\subseteq I^-,\\
        &0, &{\rm otherwise}. 
    \end{aligned}\right. \]
    Note that $\theta^{-1}(0)=Z$ by construction, and moreover $\theta(x)>0$ if and only if $x\in I^+$ and $\theta(x)<0 $ if and only if $ x\in I^-$. We call $\theta$ the \textit{oscillating function}.
\end{construction}

\begin{construction}
\label{cons:g}
    Apply \cref{thm:twisting} to $(S^2\times S^1, g_0, f)$, taking $\theta$ to be the oscillating function from \cref{cons:oscillating}. We then obtain a different Riemannian metric $g$ on $S^2\times S^1$ such that  $\M_g(p_2,p_3)\cong Z$ for the Morse function $f$.
\end{construction}

We now determine the uncompactified and compactified Morse trajectory spaces for the tuple $(S^2\times S^1,g,f)$, omitting $g$ from the notation for simplicity.

\begin{lemma}
\label{lem:12 and 23}
    For the triple $(S^2\times S^1, g, f)$, there are identifications\begin{itemize}
        \item $\overline{\M}(p_1,p_2)=\M(p_1,p_2)=\{\gamma^{\pm}\}$, 
        \item $\overline{\M}(p_3,p_4)=\M(p_3,p_4)=\{ \sigma^{\pm} \}$, 
        \item $\overline\M(p_2,p_3)=\M(p_2,p_3)
    \cong Z$.
    \end{itemize}
\end{lemma}
\begin{proof}
    By the construction of $g$, we have $g=g_0$ outside an arbitrarily small neighborhood of $L$, which proves the first two claims. For the third, we observe that there are no broken flow lines between $p_2$ and $p_3$ because $f$ is decreasing along gradient flow lines (and there are no critical points between $p_2$ and $p_3$), and hence $\overline\M(p_2,p_3)=\M(p_2,p_3)$.
\end{proof}

For the remainder of this subsection, we set \begin{align*}
    M_0 &:=(S^2\times S^1)\cap\{y=0\},\\
    L &:=(S^2\times S^1)\cap f^{-1}(0),\\
    N &:=W^u_{g_0}(p_2)\cap L=W^s_{g_0}(p_3)\cap L=(1-r)S^1\subseteq M_0\cap L,
\end{align*}
as in the proof of \cref{lem:stdgradient}. Unlike in \cref{lem:stdgradient}, there are now flow lines between $p_2$ and $p_4$, as well as between $p_1$ and $p_3$.

\begin{lemma}
\label{lem:24}
    There are homeomorphisms \[\M(p_2,p_4)\cong I^+\sqcup I^- \text{ and } \overline{\M}(p_2,p_4)\cong (I^+\cup Z)\sqcup (I^-\cup Z),\] where $I^{\pm}\cup Z\subseteq S^1$ is equipped with the subspace topology.
\end{lemma}
\begin{proof}
    As in the proof of \cref{thm:twisting}, there is a neighborhood of $N$ in $L$, denoted $N^{\prime}$, such that $(N^{\prime},N)\cong (N\times [-1,1],N\times \{0\})$. Since $N\subseteq M_0$ is determined by $\{y=0\}$, we can ensure that $N\times(0,1]\subseteq \{y>0\}$ and $N\times [-1,0)\subseteq \{y<0\}$. Recall then that 
    $$W^u(p_2)\cap L= \{(x_0,\lambda(x_0))\in N\times [-1,1] \},\quad\quad  W^s(p_3)\cap L=N=N\times\{0\},$$
    as illustrated in \cref{fig:on l 2}, where $\lambda(x)$ is a smooth function on $S^1$ such that 
    $$\lambda^{-1}(\R^+)=\theta^{-1}(\R^+)=I^+,\quad \lambda^{-1}(\R^-)=\theta^{-1}(\R^-)=I^-,\quad \lambda^{-1}(0)=\theta^{-1}(0)=Z.$$

\begin{figure}[h]
    \centering
    \input{figures/onL2}
    \caption{}
    \label{fig:on l 2}
\end{figure}
    
    Also recall that there is a neighborhood of $L$ in $S^2\times S^1$ which is diffeomorphic to the box $L\times [-1,1]$, where $L$ is identified with $L\times\{0\}$. Then the complement $(S^2\times S^1)\setminus (L\times (0,1))$ decomposes into two components $P_0\sqcup P_1$, with $\partial P_0=L\times \{-1\}$ and $\partial P_1=L\times \{1\}$ (see \cref{fig:recall}). Then, since $g=g_0$ on $P_0\sqcup P_1$, the gradient flow line for $g$ on $P_0\sqcup P_1$ is the same as the flow line for $g_0$. In particular, a gradient flow line $\gamma\in\M(p_2,p_4)$ must satisfy either $\im\gamma\cap P_i\subseteq\{y\geq 0\}$ or $\im\gamma\cap P_i\subseteq \{y\leq0\}$ for $i=0,1$.

\begin{figure}[h!]\label[figure]{fig:recall}
    \centering
    \input{figures/recall}
    \caption{}
\end{figure}

    As we showed in Lemma \ref{lem:stdgradient}, any point in $L\setminus N$ will move towards $p_4$ along the gradient flow lines, hence $\M(p_2,p_4)=(W^u(p_2)\cap L)\setminus N\cong S^1\setminus Z\cong I^+\sqcup I^-,$ where the identification $(W^u(p_2)\cap L)\setminus N\cong S^1\setminus Z$ comes from the description in terms of the image of $\lambda$.

    By definition, the compactified moduli space is \[\overline{\M}(p_2,p_4)=\M(p_2,p_4)\cup (\M(p_2,p_3)\times \M(p_3,p_4))\cong (I^+ \sqcup I^-)\cup (Z\times \{\sigma^{\pm}\}).\]
    Define a map
    \[ F\colon\overline{\M}(p_2,p_4)\to (I^+\cup Z)\sqcup(I^-\cup Z) \]
    by the evident inclusions $I^+ \sqcup I^-\subseteq (I^+\cup Z)\sqcup(I^-\cup Z) $ and $Z\times\{\sigma^{\pm}\}\hookrightarrow I^{\pm}\cup Z$. It is clear by inspection that $F$ is bijective, and (since the spaces involved are compact Hausdorff) it suffices to check that $F$ is continuous.
    
   In particular, since it is clear that $F$ is continuous on $\M(p_2, p_4)$, it suffices to show that $F(\gamma_i)\to F(\gamma)$ for every sequence of flow lines $\{\gamma_i\}_{i\geq 1}\subseteq \overline{\M}(p_2, p_4)$ that approaches a flow line $\gamma\in \M(p_2, p_3)\times \M(p_3, p_4)\cong Z\times \{\sigma^{\pm}\}$. Without loss of generality, we may assume $\gamma = (\eta, \sigma^+)$ for some $\eta\in \M(p_2, p_3)$, as the argument using $\sigma^-$ is analogous.
   
    Since $\sigma^+$ lies in $\{ y>0 \} $ and $\gamma_i$ is close to $\eta\circ \sigma^+$ for large $i$, by the distance function in \cref{def:distance fn}, we know that for large $i$ there is some point on $\im \gamma_i$ that lies in $\{ y>0 \}$, so $\im\gamma_i\cap P_0$ lies in $\{ y\geq 0 \}$ for large $i$. Then $F(\gamma_i)\in I^+\cup Z$ for large $i$, with $F(\gamma_i)=\im\gamma_i\cap L$. Since $L$ is a level hyperplane, we know that $\im \gamma_i\cap L \to \im \gamma \cap L$ according to the topology of $\overline{\M}(p_2,p_4)$ described in \cref{def:heighttop}. This shows $F(\gamma_i)=\im\gamma_i\cap L\to \im\gamma\cap L=F(\gamma)\in I^+\cup Z$, as claimed.
\end{proof}

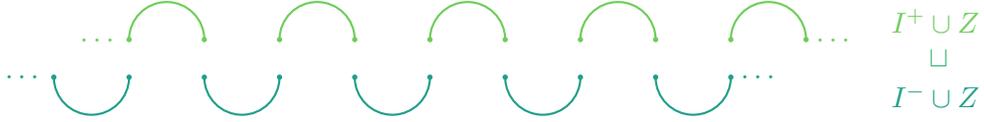
\begin{figure}[h]
    \centering
    \input{figures/M24}
    \caption{$\overline{\M}(p_2,p_4)\cong (I^+\cup Z)\sqcup(I^-\cup Z)$}
    \label{fig:placeholder}
\end{figure}

\begin{lemma}
    \label{lem:13}
    There are homeomorphisms \[\M(p_1,p_3)\cong I^-\sqcup I^+\text{ and }\overline{\M}(p_1,p_3)\cong (I^-\cup Z)\sqcup(I^+\cup Z),\] where $I^\pm\cup Z\subseteq S^1$ is equipped with the subspace topology.
\end{lemma}
\begin{proof}
    Similarly to the previous argument, we have $\M(p_1,p_3)=N\setminus (W^u(p_2)\cap L)\cong I^+\sqcup I^-$. On the other hand, the proof of \cref{thm:twisting} shows that $\im\gamma\cap (L\times \{1\})\subseteq \{y>0\}$ if and only if $\im\gamma\cap (L\times\{0\})\subseteq \{(x,y)\mid y>\lambda(x)\}$. 
    Then essentially the same proof as in \cref{lem:24} gives a homeomorphism $\overline{\M}(p_1,p_3)\cong (I^+\cup Z)\sqcup (I^-\cup Z)$.
\end{proof}

\begin{figure}[h]
    \centering
    \input{figures/M13}
    \caption{$\overline{\M}(p_1,p_3)\cong (I^+\cup Z)\sqcup(I^-\cup Z)$}
    \label{fig:13}
\end{figure}
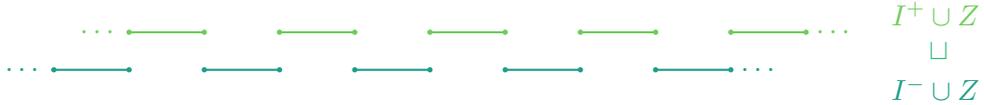

\begin{remark}\label{rmk:obs from 24 and 13}
    The following observations from \cref{lem:24,lem:13} (cf.\ \cref{fig:placeholder,fig:13}) will be helpful in the next lemma. In \cref{lem:24}:\begin{itemize}
        \item The subset $Z\subseteq I^+\cup Z$ corresponds to $Z\times \{\sigma^+\}\subseteq \overline{\M}(p_2, p_4)$. Moreover, we have $\gamma\in I^+\subseteq \M(p_2,p_4)$ if and only if $\im\gamma\cap P_0\subseteq \{y\geq 0\}$.
        \item Dually, $Z\subseteq I^-\cup Z$ corresponds to $Z\times \{\sigma^-\}\subseteq \overline{\M}(p_2, p_4)$. Moreover, we have $\gamma\in I^-\subseteq \M(p_2,p_4)$ if and only if $\im\gamma\cap P_0\subseteq \{y\leq 0\}$.
        \end{itemize}
    On the other hand, in \cref{lem:13}:
    \begin{itemize}
        \item The subset $Z\subseteq I^+\cup Z$ corresponds to $\{\gamma^-\}\times Z\subseteq \overline{\M}(p_1, p_3)$. Moreover, we have $\gamma\in I^+\subseteq \M(p_1,p_3)$ if and only if $\im\gamma\cap P_1\subseteq \{ y\leq 0 \}$; this is because $I^+\subseteq N$ lies \textit{under} the image of $y=\lambda(x)$. 
        \item Dually, $Z\subseteq I^-\cup Z$ corresponds to $\{\gamma^+\}\times Z\subseteq \overline{\M}(p_1, p_3)$. Moreover, we have $\gamma\in I^-\subseteq \M(p_1,p_3)$ if and only if $\im\gamma\cap P_1\subseteq \{ y\geq 0 \}$.
    \end{itemize}

\end{remark}

\begin{lemma}
\label{lem:14}
    There is a homeomorphism $\overline{\M}(p_1,p_4)\cong (S^1\times D^1)\sqcup \Gamma_1\sqcup \Gamma_2$, where
    \begin{itemize}
    \item $\Gamma_1:= \{(x,y)\in S^1\times \RR\mid x\in I^+\cup Z, ~0\leq y\leq \lambda(x)\}\subseteq S^1\times \R$,
    \item $\Gamma_2:=\{ (x,y)\in S^1\times \RR \mid x\in I^-\cup Z,~\lambda(x)\leq y\leq 0 \}\subseteq S^1\times \R$.
    \end{itemize}
\end{lemma}

\begin{proof}
     We separate $N\times[0,1]\subseteq L$ into four pieces along the submanifolds $W^u(p_2)$ and $W^s(p_3)$:
    \begin{itemize}
    \item[-] $\Gamma_+:=\{ (x,y)\mid x\in N,\;\max\{0,\lambda(x)\}\leq y\leq 1  \}$,
    \item[-] $\Gamma_1:=\{ (x,y)\mid x\in I^+\cup Z,\; 0\leq y\leq \lambda(x) \}$,
    \item[-] $\Gamma_2:=\{ (x,y)\mid x\in I^-\cup Z,\;\lambda(x)\leq y\leq 0 \}$,
    \item[-] $\Gamma_-:=\{ (x,y)\mid x\in N,\;-1\leq y\leq \min\{0,\lambda(x)\} \}$.
    \end{itemize}
    These four pieces are not disjoint, but their interiors are; in fact, their interiors form a partition of $N\times[0,1]\setminus(W^u(p_2)\cup W^s(p_3))$, and these spaces are respectively the closures of their interiors in $N\times [-1,1]$, as illustrated in \cref{fig:m14}.

    Now define $\Gamma_0$ to be the space obtained by gluing $\Gamma_+\sqcup \Gamma_-$ and $L\setminus (N\times (-1,1))$ along their common boundary $N\times \{1,-1\}$. Noting that $L\cong S^1\times S^1$ and $N\subseteq L$ is an equator, we see that $$\Gamma_0\cong L\setminus (N\times(-1,1))\cong S^1\times D^1.$$

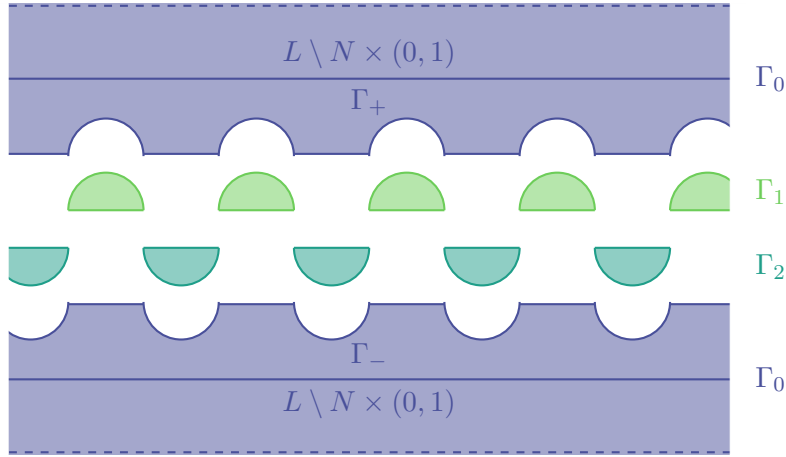
\begin{figure}[h!]
    \centering
    \input{figures/Gammas}
    \caption{$\overline{\M}(p_1,p_4)\cong \Gamma_0\sqcup \Gamma_1\sqcup \Gamma_2$}
    \label{fig:m14}
\end{figure}
    
    Note that $\M(p_1,p_4)\cong L\setminus(W^u(p_2)\cup W^s(p_3))$ is partitioned by the interiors of $\Gamma_0$, $\Gamma_1$, and $\Gamma_2$. We will construct a map
    \[ G\colon \overline{\M}(p_1,p_4)\to \Gamma_0\sqcup \Gamma_1\sqcup \Gamma_2 \]
   which is a continuous bijection; the idea is that $G$ should send $\gamma\in \overline{\M}(p_1, p_4)$ to $\im\gamma\cap L$, but we need to choose carefully which of $\Gamma_0, \Gamma_1$, or $\Gamma_2$ to send it to. Since both the source and target of $G$ are compact Hausdorff, it follows that $G$ is a homeomorphism and so together with the homeomorphism $\Gamma_0\cong S^1 \times D^1$, this will prove the lemma.

    By \cref{lem:12 and 23,lem:24,lem:13}, we have $$
    \begin{aligned}
        \overline{\M}(p_1,p_4)=\M(p_1,p_4)&\cup (\M(p_1,p_3)\times \M(p_3,p_4))\cup (\M(p_1,p_2)\times \M(p_2,p_4)) \\ &\cup (\M(p_1,p_2)\times\M(p_2,p_3)\times \M(p_3,p_4))\\=\M(p_1,p_4)&\cup (I^\pm\times\{ \sigma^\pm\})\cup(\{\gamma^\pm\}\times I^\pm)\cup(\{\gamma^\pm\}\times Z\times \{\sigma^\pm\}).
    \end{aligned}$$
    
    For $\gamma\in \M(p_1,p_4)$, $\im\gamma\cap L$ is in the interior of $L\setminus (W^u(p_2)\cup N)$ and hence is in precisely one of $\Gamma_0, \Gamma_1,$ or $\Gamma_2$. Therefore $G(\gamma):= \im\gamma\cap L$ is well-defined. Moreover, it is clear by inspection that $G$ is a homeomorphism from $\M(p_1, p_4)$ to the interior of $\Gamma_0\sqcup \Gamma_1\sqcup \Gamma_2$. 
    
    For the broken flow lines, we will associate them to a point in the boundary of $\Gamma_0\sqcup \Gamma_1\sqcup \Gamma_2$, chosen carefully to ensure that $G$ is continuous:
    \begin{itemize}
    \item if $\gamma$ is in $\{\mu^+\}\times I^+$ or $I^-\times\{ \sigma^+\}$ or $\{\mu^+\}\times Z\times\{ \sigma^+\}$, let $G(\gamma)=\im \gamma\cap L\in \Gamma_+\subseteq \Gamma_0$;
    \item if $\gamma$ is in $\{\mu^-\}\times I^+$ or $I^+\times \{\sigma^+\}$ or $\{\mu^-\}\times Z\times \{\sigma^+\}$, let $G(\gamma)=\im\gamma\cap L\in \Gamma_1$;
    \item if $\gamma$ is in $\{\mu^+\}\times I^-$ or $I^-\times\{\sigma^-\}$ or $\{\mu^+\}\times Z\times\{ \sigma^-\}$, let $G(\gamma)=\im\gamma\cap L\in \Gamma_2$;
    \item if $\gamma$ is in $\{\mu^-\}\times I^-$ or $I^+\times \{\sigma^-\}$ or $\{\mu^-\}\times Z\times\{ \sigma^-\}$, let $G(\gamma)=\im\gamma\cap L\in \Gamma_-\subseteq \Gamma_0$.
    \end{itemize}
    It is straightforward to verify that $G$, as defined above, gives a bijection between the complement $\overline{\M}(p_1, p_4)\setminus \M(p_1, p_4)$ and the boundary of $\Gamma_0\sqcup \Gamma_1\sqcup \Gamma_2$; for example,\[
    \bndry \Gamma_1 \cong I^+ \cup Z \cup I^+
    \] and $G$ (in the second bullet point above) sends $\{\mu^-\}\times I^+$ to the first copy of $I^+$, $I^+\times \{\sigma^+\}$ to the second copy of $I^+$, and $\{\mu^-\}\times Z\times \{\sigma^+\}$ to $Z$. 

    To show that $G$ is continuous, we use a similar strategy as in \cref{lem:24}. Given a sequence $\gamma_i\to \gamma$ in $\overline{\M}(p_1,p_4)$, and we need to show $G(\gamma_i)\to G(\gamma)$. Since $L$ is a level hyperplane, the assignment $\gamma\mapsto \im\gamma\cap L$ is continuous, because of the description of topology given in \cref{def:heighttop}, so we only need to assure that $G(\gamma_i)$ and $G(\gamma)$ will lie in the same target space ($\Gamma_0,\Gamma_1,\text{or }\Gamma_2$) for large $i$. Note that if $\im\gamma\cap L\notin N\times (-1,1)$, then $G(\gamma_i)$ and $G(\gamma)$ lie in $ \Gamma_0\setminus \Gamma_\pm$ for large $i$. Hence it suffices to consider the case when $\im\gamma_i\cap L,\im\gamma\cap L\in N\times[-1,1]$.
    
    To this end, let $\mathcal{N}=\{ \gamma\in \overline{\M}(p_1,p_4),\;\gamma\cap L \in N\times[-1,1] \}$ so $\gamma_i,\gamma\in \mathcal{N}$. We define
    \begin{itemize}
        \item[-] $\mathcal{N}_{++}=\{\gamma\in \mathcal{N}\mid \im\gamma\cap P_1\subseteq\{y\geq0\}\text{ and }\im \gamma\cap P_0\subseteq\{y\geq 0\}\}$,
        \item[-] $\mathcal{N}_{-+}=\{\gamma\in \mathcal{N}\mid \im\gamma\cap P_1\subseteq\{y\leq0\}\text{ and }\im\gamma\cap P_0\subseteq\{y\geq 0\}\}$,
        \item[-] $\mathcal{N}_{+-}=\{\gamma\in \mathcal{N}\mid \im\gamma\cap P_1\subseteq\{y\geq0\}\text{ and }\im\gamma\cap P_0\subseteq\{y\leq 0\}\}$,
        \item[-] $\mathcal{N}_{--}=\{\gamma\in \mathcal{N}\mid \im\gamma\cap P_1\subseteq\{y\leq0\}\text{ and }\im\gamma\cap P_0\subseteq\{y\leq 0\}\}$,
    \end{itemize}
    and note that these four subsets are disjoint and open in $\mathcal N$. Moreover, by \cref{rmk:obs from 24 and 13}, we have \[\mathcal{N}=\mathcal{N}_{++}\cup\mathcal{N}_{-+}\cup\mathcal{N}_{+-}\cup\mathcal{N}_{--}.\]
    \cref{rmk:obs from 24 and 13} also implies $G(\mathcal{N}_{++})\subseteq \Gamma_+$, $G(\mathcal{N}_{-+})\subseteq \Gamma_1$, $G(\mathcal{N}_{+-})\subseteq \Gamma_2$ and $G(\mathcal{N}_{--})\subseteq \Gamma_-$. Because of the distance function described in \cref{def:distance fn}, we know that $\gamma_i$ will lie in the same component $\mathcal{N}_{\pm\pm}$ as $\gamma$, for large $i$, it then follows that $G(\gamma_i)$ and $G(\gamma)$ will lie in the same component of $\Gamma_{+}\sqcup \Gamma_-\sqcup \Gamma_0\sqcup \Gamma_1$, as desired. 
\end{proof}

Now that we have described the compactified Morse trajectory spaces of $(S^2\times S^1,g, f)$, we can prove \cref{thm:counter}.

\begin{proof}[Proof of \cref{thm:counter}]
The flow category associated to $(S^1\times S^2, g, f)$ is the topological category $\cat C_f$ with four objects, the critical points $\{p_1,p_2,p_3,p_4\}$, and whose homspaces are the compactified Morse trajectory spaces between these critical points. Note that $\cat{C}_f(p_i,p_j)\neq\varnothing$ if and only if $i\leq j$; hence any chain of nontrivial (non-identity) morphisms has length at most $3$. Recall that $B\cat C_f$ is the classifying space of this topological category; in fact, we may realize $B\cat C_f$ as the geometric realization of the semi-simplicial space $\Lambda_\bullet$ given by
    \[\begin{aligned}
        \Lambda_0&=\{p_1,p_2,p_3,p_4\}, \\
        \Lambda_1&=\bigsqcup_{1\leq i<j\leq 4}\overline{\M}(p_i,p_j), \\
        \Lambda_2&=\left(\overline{\M}(p_1,p_2)\times \overline{\M}(p_2,p_3)\right) \sqcup \left(\overline{\M}(p_1,p_2)\times \overline{\M}(p_2,p_4)\right)\\
        &\phantom{...}\sqcup \left(\overline{\M}(p_1,p_3)\times \overline{\M}(p_3,p_4)\right)
        \sqcup \left(\overline{\M}(p_2,p_3)\times \overline{\M}(p_3,p_4)\right), \\
        \Lambda_3&= \overline{\M}(p_1,p_2)\times \overline{\M}(p_2,p_3)\times \overline{\M}(p_3,p_4)
    \end{aligned} 
    \] and whose higher simplices are all degenerate, with face maps as in $N\cat C_f$. We will show that $H_3(\abs{\Lambda}) = 0$, so $B\cat C_f$ cannot be weakly homotopy equivalent to $S^2\times S^1$.

To compute the homology of $\abs{\Lambda_\bullet}$, we may apply the classical geometric realization spectral sequence on the semi-simplicial space $\Lambda_\bullet$ (see \cite[\S 1.4]{ERW}), to obtain
    \[ E^1_{r,s}=H_r(\Lambda_s)\Longrightarrow H_{r+s}(B\cat{C}_f) \]
    with $d^1=\sum_{i}(-1)^i(d_i)_*$.

    \cref{lem:12 and 23,lem:24,lem:13,lem:14}, describe the homeomorphism types of the compact moduli spaces $\overline{\M}(p_i,p_j)$, and we observe that the weak homotopy types of these spaces are quite simple. In fact, except for $\overline{\M}(p_1,p_4)$, the compactified moduli spaces are weakly equivalent to discrete sets (since all the path components of the spaces are either a point or an interval). The space $\overline{\M}(p_1,p_4)$ has a path-component homeomorphic to $S^1\times D^1$, and all the other path-components are either a point or a disk, hence $\overline{\M}(p_1,p_4)$ is weakly equivalent to a disjoint union of $S^1$ and a discrete set. Hence we have $H_r(\Lambda_s)=0$ for $r>0$ except $H_1(\Lambda_1)=H_1(\overline{\M}(p_1,p_4))=\mathbb{Z}$. Therefore, in order to show that $H_3(B\cat{C}_f)=0$, it suffices to show that $\ker (d^1)=0$ in $H_0(\Lambda_3)$.

    \begin{sseqdata}[name = basic, homological Serre grading, classes={draw=none}, xscale=2.5, y axis gap=25pt, x axis gap =15pt, x label= {$s$}, y label={$r$}, x label style={xshift=132pt, yshift=30pt}, y label style={rotate=270, xshift=-2pt, yshift=-25pt}, x axis tail = 15 pt, y axis tail=15pt, x tick gap=10pt, y tick gap=10pt]
\class["H_0(\Lambda_0)"](0,0)
\class["H_0(\Lambda_1)"](1,0)
\class["H_0(\Lambda_2)"](2,0)
\class["H_0(\Lambda_3)"](3,0)
\class["\ZZ"](1,1)
\d1(1,0)
\d1(2,0)
\d["d^1"]1(3,0)
\end{sseqdata}
\begin{center}
    \printpage[ name = basic, page = 1 ]
\end{center} 

We now detail $d^1\colon H_0(\Lambda_3)\to H_0(\Lambda_2)$, making use of the isomorphism $H_0(\Lambda_s)\cong \ZZ[\pi_0\Lambda_s]$.
By \cref{lem:12 and 23}, we have  $\pi_0\Lambda_3=\{\gamma^\pm\}\times Z\times \{\sigma^\pm\}$. On the other hand,
    \begin{equation*}
    \begin{aligned}
        \pi_0\Lambda_2 &= \pi_0(\overline{\M}(p_1,p_2)\times \overline{\M}(p_2,p_3))\sqcup \pi_0(\overline{\M}(p_1,p_2)\times \overline{\M}(p_2,p_4)) \\
        &\quad \sqcup \pi_0(\overline{\M}(p_1,p_3)\times \overline{\M}(p_3,p_4))\sqcup \pi_0(\overline{\M}(p_2,p_3)\times \overline{\M}(p_3,p_4))\\
        &\cong (\{\gamma^{\pm}\}\times Z) \sqcup (\{\gamma^{\pm}\}\times \pi_0 \overline{\M}(p_2,p_4))\\
        &\quad \sqcup (\pi_0 \overline{\M}(p_1,p_3) \times \{\sigma^{\pm}\}) \sqcup (Z\times \{\sigma^{\pm}\}).
    \end{aligned}
    \end{equation*}
Thus, for instance, for $(\gamma^+,x,\sigma^+)\in \pi_0\Lambda_3$, we have
\begin{equation}\label{eqn:d1 simple}
    d^1(\gamma^+,x,\sigma^+)=(x,\sigma^+)-(\gamma^+\circ x,\sigma^+)+(\gamma^+,x\circ \sigma^+)-(\gamma^+,x)\in \pi_0\Lambda_2.
\end{equation}  
Write $x_n=\frac{1}{n}$ for $n\geq 1$ and $x_0=0$, so then $Z=\{x_k\mid k\geq0\}$. Suppose $z\in H_0(\Lambda_3)$ such that $d^1(z)=0$. We may express \[ z=\sum_{i\geq 0}a_i(\gamma^+,x_i,\sigma^+)+b_i(\gamma^+,x_i,\sigma^-)+c_i(\gamma^-,x_i,\sigma^+)+d_i(\gamma^-,x_i,\sigma^-) \]
    for some $a_i,b_i,c_i,d_i\in \ZZ$, $i\geq 0$, only finitely many of which are nonzero. Applying $d^1$ to $z$, we see the component of $d^1(z)$ in $Z\times\{\sigma^{\pm}\}$ is 
\[ \sum_{i\geq 0}(a_i+c_i)(x_i,\sigma^+)+(b_i+d_i)(x_i,\sigma^-)=0. \]
    Hence $a_i+c_i=b_i+d_i=0$ for all $i\geq 0$. Similarly, the component of $d^1(z)$ in $\{\gamma^{\pm}\}\times Z$ is 
    \[ -\sum_{i\geq 0}(a_i+b_i)(\gamma^+,x_i)+(c_i+d_i)(\gamma^-,x_i)=0. \]
    Hence $a_i+b_i=c_i+d_i=0$ for all $i\geq 0$. This implies that for each $i\geq 0$ there exists $\lambda_i\in \ZZ$ such that $(a_i,b_i,c_i,d_i)=(\lambda_i,-\lambda_i,-\lambda_i,\lambda_i)$. So if we let
    \[ \langle x_i\rangle=(\gamma^+,x_i,\sigma^+)-(\gamma^+,x_i,\sigma^-)-(\gamma^-,x_i,\sigma^+)+(\gamma^-,x_i,\sigma^-)\in H_0\Lambda_3 \]
    then we may write $z=\sum_{i\geq 0}\lambda_i\langle x_i\rangle$ and $d^1(z) = \sum_{i\geq 0} \lambda_i d^1\gen{x_i}$. Using a similar computation as in \cref{eqn:d1 simple}, we have
    \begin{equation}
    \label{equation:counter}
    \begin{aligned}
        d^1\langle x_i\rangle=&-(\gamma^+\circ x_i,\sigma^+)+(\gamma^-\circ x_i,\sigma^+)+(\gamma^+\circ x_i,\sigma^-)+(\gamma^-,x_i\circ \sigma^-) \\
        &+(\gamma^+,x_i\circ \sigma^+)-(\gamma^-,x_i\circ\sigma^+)-(\gamma^-\circ x_i,\sigma^-)-(\gamma^+,x_i\circ \sigma^-),
    \end{aligned}
    \end{equation}
    with all eight terms lying in different components of $\Lambda_2$. We claim that $\lambda_i=0$ for all $i$. Given this, we then have $z=0$ and so $\ker(d^1)=0$. By the spectral sequence, we may then conclude $H_3(B\cat C_f)=0$, as claimed.

    To show that $\lambda_i=0$ for all $i\geq 0$, we will show that if $\lambda_j\neq 0$ for some $j\geq 0$, then $\lambda_{j+1}\neq 0$.  Since $\lambda_i\neq 0$ for only finitely many $i$, by assumption, we must then have $\lambda_i=0$ for all $i\geq 0$. 

    We first consider the case $j>0$, so suppose $\lambda_j\neq0$ for some $j>0$. Then $\overline{I_j}=[x_{j+1},x_j]$ is a path-component of either $I^+\cup Z$ or $I^-\cup Z$. We may assume $\overline{I_j}\subseteq I^+\cup Z$,   
    as the argument is the same for $\overline{I_j}\subseteq I^-\cup Z$.
    In this case, we have that \[(\gamma^+, \overline{I_j})\in \{\gamma^{+}\}\times \pi_0( (I^+\cup Z))\subseteq \pi_0(\overline{\M}(p_1,p_2)\times \overline{\M}(p_2,p_4)).\] Then the summand of $d^1(z)$ in $\ZZ[(\gamma^+, \overline{I_j})]$ is $(\lambda_j+\lambda_{j+1})[(\gamma^+,\overline{I_j})]$ since, by \cref{equation:counter}, $d^1\langle x_i\rangle$ will contribute to this summand if and only if $x_i\in\overline{I_j}$, i.e. $i=j$ or $j+1$. Therefore we may conclude that $\lambda_j+\lambda_{j+1}=0$ and hence $\lambda_{j+1}\neq 0$. 

    For $j=0$, we may use the same argument by observing that $\overline{I_0}=(\infty,0]\cup [1,\infty)\cup \{\infty\}$ is a path-component of $I^+\cup Z$. So if we consider the summand of $d^1(z)$ in $Z[(\gamma^+,\overline{I_0})]$, we can deduce that $\lambda_0+\lambda_1=0$ and hence if $\lambda_0\neq 0$ then $\lambda_1\neq 0$ as well.
\end{proof}

\begin{remark}
    When $W^u(p_2)\cap L$ and $W^s(p_3)\cap L=N$ intersect transversally, there are only finitely many intersection points. In the case of \cref{fig:transversal}, where there are four intersection points $x,y,z,w$, we can verify that $ \langle x\rangle-\langle y\rangle+\langle z\rangle-\langle w\rangle$ is a generator of $\ker(d^1)$ and $H_3(B\cat{C}_f)\cong \ZZ$ as expected.
\end{remark}

\begin{figure}[h]
    \centering
    \includegraphics[width=\linewidth]{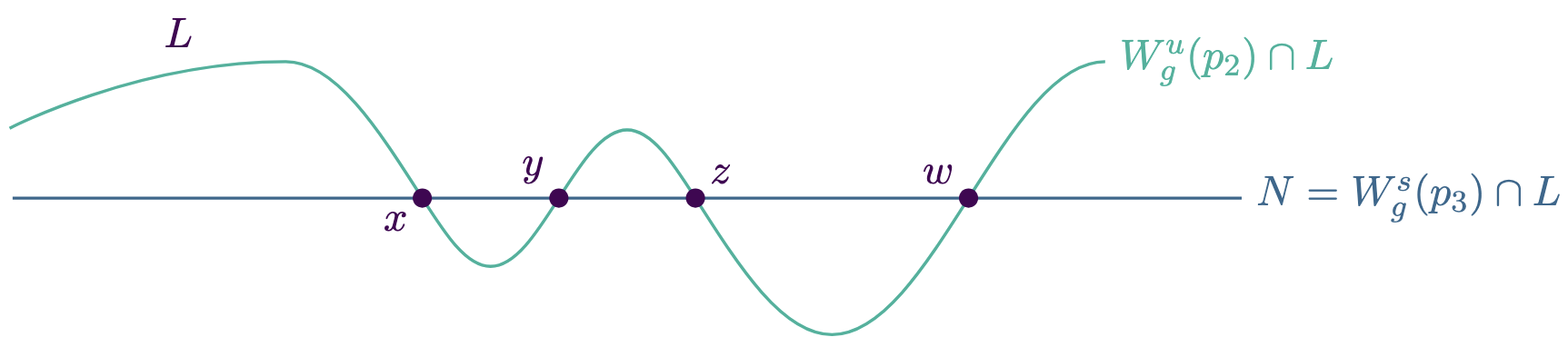}
    \caption{}
    \label{fig:transversal}
\end{figure}

\appendix
\section{Background on classifying spaces of topological categories}\label{sec:topcats}

In this appendix, we provide some background on simplicial spaces, particularly an overview of the material necessary in \cref{sec:thm} regarding classifying spaces of topological categories and some lemmas used in the proof of \cref{thm:loc}. For further details on these topics, we point the reader to \cite{goerss/jardine:1999,segal:1968}. Throughout, we let $\Set$ denote the category of sets and functions and $\Top$ denote a convenient category of topological spaces (such as $k$-spaces, as in \cite{segal:1968}).

\begin{definition}
The \textit{simplex category} ${\Delta}$ is the category whose objects are finite, non-empty ordinals \[
[n] = \{0,1,\dots,n\}
\]
and whose morphisms are order-preserving maps. A \textit{simplicial space} is a functor $X_*\colon \Delta^{\op}\to \Top$. As is standard, we write $X_n$ for the space $X[n]$, and call it the space of \textit{$n$-simplices}. A morphism of simplicial spaces is a natural transformation, and we denote the category of simplicial spaces by $s\Top$. Similarly, a \textit{simplicial set} is a functor $X_*\colon \Delta^{\op}\to \Set$, and $s\Set$ denotes the category of simplicial sets and natural transformations between them.
\end{definition}

Note that every simplicial set may be viewed as a simplicial space via the inclusion $\Set\subseteq \Top$ that gives a set the discrete topology. To specify a simplicial space, it suffices to supply the data of spaces of $n$-simplices $X_n$ for $n\geq 0$, along with \textit{face} and \textit{degeneracy} maps \[
d_i\colon X_n\to X_{n-1} ~\text{ and }~ s_j\colon X_n\to X_{n+1},
\] for $0\leq i,j\leq n$ which satisfy the \textit{simplicial relations} (see \cite[Equation (1.3)]{goerss/jardine:1999}). A simplex in $X_n$ is called \textit{degenerate} if it is the image of some $s_j$, and is \textit{non-degenerate} otherwise.

The geometric realization functor provides a way to turn a simplicial space $X_*$ into a topological space, essentially by forming a simplicial complex which has an $n$-cell for each $n$-simplex in $X_n$, glued together according to the face maps of $X_*$.

\begin{definition}\label{defn:geom real}
The \textit{geometric realization} of a simplicial space $X$ is the colimit \[
\abs{X} = \colim \left(\coprod_{f\colon [n]\to [m]} X_m \times {\Delta^n} ~ \overset{f_*}{\underset{f^*}\rightrightarrows} ~ \coprod_{[n]} X_n \times {\Delta^n} \right)
\] in $\Top$. The map $f_* \colon X_m \times {\Delta^n} \to X_m \times {\Delta^m}$ includes faces by  $(x_1,\dots, x_n)\mapsto (y_1,\dots, y_m)$ for\[
y_i = \left\{\begin{array}{cc}
     0 &  f^{-1}(i)=\emptyset\\
     \sum_{j\in f^{-1}(i)} x_j & \text{otherwise.}
\end{array}\right.
\] The map $f_* \colon X_m \times {\Delta^n} \to X_n \times {\Delta^n}$ collapses degeneracies via the simplicial structure $Xf \colon X_m \to X_n$. 
\end{definition}
In practice, it is often useful to use a more concrete description of the geometric realization, which is given by
\[
\coprod_{n\geq 0} X_n\times \Delta^n/\sim
\]
where $(x,s^jy) \sim (s_jx,y)$ and $(x,d^iy) \sim (d_ix,y)$. The $d^i$ inserts a $0$ in the $i^{th}$ coordinate and the $s^j$ adds the $x_j$ and $x_{j+1}$ coordinates. Geometrically, $d^i$ inserts ${\Delta^{n-1}}$ as the $i^{th}$ face of $\Delta^n$ and $s^j$ projects $\Delta^{n+1}$ onto the topological $n$-simplex orthogonal to its $j^{th}$ face.

\begin{remark}\label{rmk:realization of sSet is CW}
The geometric realization of a simplicial \textit{set} is always a CW complex \cite[Proposition 2.3]{goerss/jardine:1999}, but the geometric realization of a simplicial space may not be. For example, if $A$ is any space which is not a CW complex, then the geometric realization of the constant simplicial space $X_n=A$ (where all the face and degeneracy maps are identities) is again $A$.
\end{remark}

We now prove some key lemmas that are used in the proof of \cref{thm:loc}. Namely, that geometric realization preserves the properties of being locally contractible (\cref{def:loccon}) and being homotopy equivalent to a CW complex. The following definitions will be helpful in the proofs of both lemmas.

\begin{definition}
If $X_*$ is a simplicial space, then the \textit{$n$-skeleton} of $\abs{X}$ is defined by\[
    \sk_n\abs{X} :=\Big(\coprod_{0\leq k\leq n}X_k\times \Delta^k\Big) / \sim \]
    for $n\geq 0$, where $\sim$ denotes the same equivalence relation as in the geometric realization.
\end{definition}

Note that $\abs{X} = \colim_n \sk_n\abs{X} = \bigcup_n \sk_n\abs{X}$. It is often convenient to prove things about $\abs{X}$ via inductive arguments on the skeleta. For our purposes, we need an additional technical condition on our simplicial spaces to ensure we can run such inductive arguments.

\begin{definition}
\label{def:splitsim}
We say that $X_*$ is a \textit{split} simplicial space if for each $n\geq0$, there is a decomposition of topological spaces
\[ X_n=X_n^{\nd} \sqcup X_n^{\rm deg} \]
where $X_n^{\nd}, X_n^{\rm deg}\subseteq X_n$ are the subspaces of non-degenerate and degenerate simplices, respectively.

\end{definition}

For a split simplicial space $X_*$, there is a pushout diagram
\begin{equation}\label{eqn:skel pushout}
    \begin{tikzcd}
	{X_n^{\nd}\times \partial\Delta^n} & {\sk_{n-1}\abs{X}} \\
	{X_n^{\nd}\times \Delta^n} & {\sk_n\abs{X}}
	\arrow["f_n", from=1-1, to=1-2]
	\arrow[hook, from=1-1, to=2-1]
	\arrow[from=1-2, to=2-2]
	\arrow["g_n", swap, from=2-1, to=2-2]
    \end{tikzcd}
\end{equation}
where $f_n$ maps $X_n^{\nd}\times \bndry \Delta^n$ to the $(n-1)$-skeleton via the face maps.

\begin{lemma}
\label{lem:simloc}
    Suppose $X_*$ is a split simplicial space such that $X_n$ is compact Hausdorff and locally contractible for each $n\geq 0$. Then the geometric realization $\abs{X}$ is locally contractible.
\end{lemma}
\begin{proof}
 Since $\abs{X_*}=\colim_n \sk_n\abs{X}=\bigcup_n \sk_n\abs{X}$, a subset $U\subseteq X$ is open if and only if $U\cap \sk_n\abs{X}$ is open in $\sk_n\abs{X}$ for all $n$. Let $x\in V$ for $V\subseteq \abs{X}$ open and set $V_n:= V\cap \sk_n\abs{X}$. Let $r\geq 0$ be minimal such that $x\in \sk_r\abs{X}$ (i.e. $x\not\in \sk_k\abs{X}$ for all $k<r$); therefore we may view $x$ as an element of $X_r^\nd\times (\Delta^r\setminus \bndry \Delta^r)\subseteq \sk_r\abs{X}$. We will construct an open subset $U\subseteq V$  as the colimit of a sequence \[
    U_r \hookrightarrow U_{r+1} \hookrightarrow \dots
    \] so that $U_r\hookrightarrow V_r\hookrightarrow V$ is null-homotopic and $U_{n+1}\subseteq V_{n+1}$ is the mapping cylinder of $U_n$ for all $n\geq r$. This then implies that $U\hookrightarrow V$ is null-homotopic, and hence that $\abs{X}$ is locally contractible. 
    
    To construct the $U_n$, we use the inductive description of the $n$-skeleton as a pushout from \ref{eqn:skel pushout}. In particular, we see that $U\subseteq \sk_n\abs{X}$ is open if and only if $U\cap \sk_{n-1}\abs{X}$ is open in $\sk_{n-1}\abs{X}$ and $g_n^{-1}(U)$ is open in $X_n^{\nd}\times \Delta^n$. 

    Since $X_r$ is locally contractible, $X_r^{\nd}$ as well as the product $X_r^{\nd}\times (\Delta^r\setminus \bndry \Delta^r)$ are also locally contractible and so we may find open subsets $W_r\subseteq X_r^\nd$ and $Z_r\subseteq \Delta^r\setminus \bndry\Delta^r$ such that $x\in W_r\times Z_r\subseteq \overline{W_r\times Z_r}\subseteq V_r$ and $W_r\times Z_r\hookrightarrow V_r$ is null-homotopic. We set $U_r:= W_r\times Z_r$. We will define $U_n$ inductively, for $n\geq r$, so suppose that we have constructed $U_{n-1}\subseteq Y_{n-1}$, satisfying the desired hypotheses, with $x\in U_{n-1}\subseteq \overline{U_{n-1}}\subseteq V_{n-1}$ for some $n>r$.

    To construct $U_n$, we will make use of the bijection\begin{equation}
    \{\text{open sets in}\,\,\sk_n\abs{X}\}\xrightarrow{\;\;\sim\;\;}
    \left\{
    \begin{aligned}
     (V,W):\;&V\;\text{is open in}\;\sk_{n-1}\abs{X},\; W\;\text{is open in }\,X_n^\nd\times \Delta^n,\; \\ &W\cap (X_n^\nd\times \partial \Delta^n)=f_n^{-1}(V).
    \end{aligned}
    \right\}
    \label{eqn:op}
    \end{equation}
    given by $U\mapsto(U\cap \sk_{n-1}\abs{X},g_n^{-1}(U))$. To see that this is indeed a bijection, we note that for subsets $V\subseteq \sk_{n-1}\abs{X}$ and $W\subseteq X_n^\nd\times \Delta^n$, the pushout diagram above implies that there exists $Z\subseteq \sk_n\abs{X}$ with $V=Z\cap \sk_{n-1}\abs{X}$ and $W=g_n^{-1}(Z)$ if and only if $W\cap (X_n^\nd\times \partial \Delta^n)=f_n^{-1}(V)$.
    
    We now produce an open subset $W_n\subseteq X_n^\nd\times \Delta^n$ so that $W_n\cap (X_n^\nd\times \bndry \Delta^n) = f_n^{-1}(U_{n-1})$. First note that $f_n^{-1}(V_{n-1})= g_n^{-1}(V_n)\cap (X_n^\nd\times \partial \Delta^n)$. Now choose $\varepsilon>0$ so that the product $X_n^\nd\times \bndry \Delta^n \times [0,\varepsilon]\subseteq X_n^\nd \times \Delta^n$ defines a neighborhood of $X_n^\nd\times \bndry \Delta^n$ and set \[\tilde{V}_n := g_n^{-1}(V_n)\cap (X_n^\nd\times \bndry \Delta^n \times [0,\varepsilon])\subseteq X_n^\nd\times \Delta^n.\]  Since $f_n^{-1}(\overline{U_{n-1}})\subseteq \tilde{V}_n$ is compact, we can find $\tilde{\varepsilon}>0$ so that \[
    W_n:= f_n^{-1}(U_{n-1})\times [0,\tilde{\varepsilon})
    \] is an open subset of $\tilde{V}_n$. Note that $W_n\cap (X_n^\nd\times\partial \Delta^n)=f_n^{-1}(U_{n-1})$, and so by the correspondence \ref{eqn:op}, the subset \[
    U_n := W_n\bigcup_{f_n^{-1}(U_{n-1})}U_{n-1}\subseteq \sk_n\abs{X}
    \] is open. By the definition of $W_n$ (and consequently $U_n$), it is straightforward to verify that $x\in U_n\subseteq \overline{U}_n\subseteq V_n$ and $U_n$ is a mapping cylinder of $U_{n-1}$, as desired. 
\end{proof}

\begin{lemma}
    \label{lem:sep}
    Suppose $X_*$ is a split simplicial space such that $X_n$ is compact Hausdorff and second countable for each $n\geq 0$. Assume that $X_*$ is finite dimensional, i.e. the geometric realization $\abs{X}=\sk_N\abs{X}$ for some large enough $N$. Then $\abs{X}$ is second countable.
\end{lemma}
\begin{proof}
Similarly as in the previous lemma, we consider the inductive description of $\abs{X}$ via its $n$-skeleton from \ref{eqn:skel pushout}.
Since each $X_n$ is compact Hausdorff and second countable, we know $X_n^\nd$ is also compact Hausdorff and second countable by the split condition. We proceed by induction on $n$. For $n=0$, the space $\sk_0|X|=X_0$ is second countable by assumption. Now suppose $\sk_{n-1}|X|$ is second countable for some $n>0$, and let $\{ U_j \}_{j\in \NN}$ be a countable basis for $\sk_{n-1}|X|$ and $\{ V_l \}_{l\in \NN}$ a countable basis for $X_n^\nd\times (\Delta^n\setminus\partial\Delta^n)$. 

Choose $\varepsilon>0$ and an embedding $X_n^\nd\times \partial \Delta^n\times [0,\varepsilon] \subseteq X_n^\nd\times \Delta^n$ as in the previous lemma. Then, as in \cref{lem:simloc}, we may construct open sets \[
U_j'(\tilde{\varepsilon}) = \left(f_n^{-1}(U_j)\times [0,\tilde{\varepsilon})\right)\bigcup_{f_n^{-1}(U_j)}U_j \subseteq \sk_n|X|,
\] for $\tilde{\varepsilon}\in \QQ\cap (0,\varepsilon)$, using the description of open subsets of $\sk_n |X|$ given previously. Then, by essentially the same proof as in the previous lemma, we see that 
\[ \left\{ U'_j(\tilde{\varepsilon}) \mid \;\tilde{\varepsilon}\in \QQ\cap (0,\varepsilon), j\in \NN\right\} \;\bigcup\; \{ V_l, l\in \NN \}\]
is a countable basis for $\sk_n|X|$. Since $\abs{X}$ is finite-dimensional, the claim follows.
\end{proof}

\begin{lemma}
\label{lem:scw}
Suppose $X_*$ is a split simplicial space such that $X_n$ is homotopy equivalent to a CW complex for each $n\geq 0$. Then the geometric realization $|X|$ is homotopy equivalent to a CW complex.
\end{lemma}

\begin{proof}
We know from the split condition that $X_n^\nd$ is homotopy equivalent to a CW complex for all $n$. In the pushout diagram \ref{eqn:skel pushout}, since $X_n^\nd\times\partial \Delta^n\hookrightarrow X_n^\nd\times \Delta^n$ is a cofibration, the inclusion $\sk_{n-1}\abs{X}\hookrightarrow \sk_n\abs{X}$ is as well. Moreover, since $|X_*|=\bigcup_n\sk_n\abs{X}$, it suffices to show that each $\sk_n\abs{X}$ is homotopy equivalent to a CW complex, $n\geq 0$.

By assumption, $X_0\cong X_0^\nd$ is homotopy equivalent to a CW complex. So now suppose that $\sk_{n-1}\abs{X}$ is homotopy equivalent to a CW complex for some $n \geq 0$; we claim that $\sk_n\abs{X}$ is homotopy equivalent to a CW complex. The pushout diagram \ref{eqn:skel pushout} is a homotopy pushout diagram because $X^{\nd}\times \partial \Delta^n\hookrightarrow X^\nd \times \Delta^n$ is a cofibration, so we may assume that $X_n^\nd$ and $\sk_{n-1}\abs{X}$ are actual CW complexes. Now, using \cite[Proposition 0.18]{Hatcher}, we can replace $f_n$ by a homotopic cellular map, so that the pushout $\sk_n\abs{X}$ is a CW complex.
\end{proof}

For our purposes, the most important examples of simplicial spaces come from nerves of topological categories. By topological category, we mean category internal to $\Top$ (in the sense of \cite[\S XII.1]{maclane}).

\begin{definition}\label{def:topcat}
    A \textit{topological category} $\cat C$ is given by a space of objects $\cat C_0$ and a space of morphisms $\cat C_1$ with four continuous structure maps\begin{center}
    \begin{tikzcd}
    \cat C_0 \ar[rr, "i"] &&  \cat C_1 \arrow[ll, bend left=30, "s"] \ar[ll, bend right=30, swap, "t"]
\end{tikzcd},\hspace{1cm}\begin{tikzcd}
    \cat C_1 \times_{\cat C_0} \cat C_1 \ar[r, "\circ"] & \cat C_1
    \end{tikzcd},
\end{center} 
which we call \textit{source} $s$, \textit{target} $t$, \textit{identity} $i$, and \textit{composition} $\circ$. These structure maps satisfy the usual commutative diagrams that make $\cat C$ a category, i.e.\ that $i$ and $\circ$ have the expected domain and codomain,
\begin{center}
    \begin{tikzcd}
    \cat C_0 \ar[r, "i"]\ar[dr, equal] & \cat C_1 \ar[d, shift left, "t"] \ar[d, shift right, swap, "s"]\\
    ~ & \cat C_0
    \end{tikzcd},\hspace{.5cm}\begin{tikzcd}
    \cat C_1 \ar[d, swap, "t"]
    & \cat C_1\times_{\cat C_0}\cat C_1 \ar[d, "\circ"] \ar[l, swap, "\pi_1"] \ar[r, "\pi_2"]
    & \cat C_1 \ar[d, "s"]\\
    \cat C_0 
    & \cat C_1 \ar[l, "t"] \ar[r, swap, "s"]
    & \cat C_0
    \end{tikzcd},
\end{center}
and that composition is unital (with identity $i$) and associative,
\begin{center}
    \begin{tikzcd}
    \cat C_0\times_{\cat C_0}\cat C_1 \ar[d, swap, "\pi_2"] \ar[r, "i\times \id"]
    & \cat C_1\times_{\cat C_0}\cat C_1 \ar[d, "\circ"] 
    & \cat C_1\times_{\cat C_0}\cat C_0 \ar[d, "\pi_1"] \ar[l, swap, "\id \times i"]
    \\
    \cat C_1 \ar[r, equal] & \cat C_1 \ar[r, equal] & \cat C_1 
    \end{tikzcd} 
    
    \vspace{.2cm} 
    
    \begin{tikzcd}
    \cat C_1\times_{\cat C_0} \cat C_1\times_{\cat C_0}\cat C_1  \ar[d, swap, "\id\times \circ"] \ar[r, "\circ\times \id"]
    & \cat C_1\times_{\cat C_0}\cat C_1 \ar[d, "\circ"]
    \\
    \cat C_1\times_{\cat C_0}\cat C_1 \ar[r, swap, "\circ"]
    & \cat C_1
    \end{tikzcd}.
\end{center}
A \textit{continuous functor} $F\colon \cat C\to \cat D$ between two topological categories consists of two continuous maps, \[
F_0\colon \cat C_0 \to \cat D_0 ~\text{ and }~ F_1\colon \cat C_1 \to \cat D_1,
\] which are compatible with the four structure maps in the usual way. That is, such that the following diagrams commute\[
    \begin{tikzcd}
    \cat C_1\times_{\cat C_0} \cat C_1\ar[d, swap, "\circ_{\cat C}"] \arrow[r, "F_1\times F_1"] & \cat D_1\times_{\cat D_0} \cat D_1 \ar[d, "\circ_{\cat D}"]\\
    \cat C_1 \ar[r, swap, "F_1"] & \cat D_1
    \end{tikzcd}\hspace{.5cm}\begin{tikzcd}
    \cat C_1 \ar[d, swap, "F_1"] \arrow[r, shift right, swap, "t"]\ar[r, shift left, "s"] & \cat C_0 \ar[d, "F_0"]\ar[r, "i"] & \cat C_1\ar[d, "F_1"]\\
    \cat D_1\ar[r, shift right, swap, "t"] \ar[r, shift left, "s"] & \cat D_0\ar[r, swap, "i"] & \cat D_1
    \end{tikzcd}.
\] We let $\Cat(\Top)$ denote the category of topological categories and continuous functors.
\end{definition}

Recall that if $\cat C$ is an ordinary (non-topological) category, its nerve $N_*\cat C$ is a simplicial set whose $0$-simplices are the objects of $\cat C$ and whose $n$-simplices are strings of $n$ composable morphisms. 
If $\cat C$ if a topological category, then its nerve inherits a topology as well.
For topological categories, the nerve is essentially the same, but now the collection of $n$-composable morphisms forms a space rather than just a set.

\begin{definition}\label{def:topnerve}
    The \textit{nerve} of a topological category is the simplicial space $N_*\cat C$ with $N_0\cat C=\ob\cat C$ and, for $n\geq 1$,
\[
N_n\cat C = \cat C_1\times_{\cat C_0} \cat C_1\times_{\cat C_0}\dots \times_{\cat C_0} \cat C_1,
\] where each $\cat C_1\times_{\cat C_0}\cat C_1$ is the pullback of $\cat C_1\xrightarrow{t} \cat C_0 \xleftarrow{s} \cat C_1$. An element of $N_n\cat C$ is a length $n$ sequence of composable morphisms\[
c_0\xrightarrow{f_1} \dots \xrightarrow{f_n} c_n.
\] For $0<i<n$, the face map $d_i\colon N\cat C_n\to N\cat C_{n-1}$ returns the string of $n-1$ composable arrows\[
c_0\xrightarrow{f_1} c_1\to \dots \to c_{i-1}\xrightarrow{f_{i+1}\circ f_i} c_{i+1}\to \dots \xrightarrow{f_n} c_n.
\] In the cases that $i=0,n$, we instead omit that $i^{th}$ arrow. At level $n=1$, we have only two face maps $d_0,d_1\colon N\cat C_1\to N\cat C_0$ given by the source and target maps, respectively. 
The degeneracy map $s_i\colon N\cat C_n\to N\cat C_{n+1}$ returns the string of $n+1$ composable arrows\[
c_0\xrightarrow{f_1} c_1\to \dots \to c_{i-1} \xrightarrow{f_{i}} c_i\xrightarrow{\id_{c_i}}c_i \xrightarrow{f_{i+1}} c_{i+1}\to \dots \xrightarrow{f_n} c_n.
\] 
These face and degeneracy maps are evidently continuous, as they are defined using the composition and projection maps. A continuous functor $\cat C\to \cat D$ induces a map $N\cat C\to N\cat D$, and this defines a functor $N\colon {\Cat(\Top)}\to {s\Top}$.
\end{definition}

 Post-composing with geometric realization, we obtain the classifying space functor.

 \begin{definition}
     The \textit{classifying space} of a topological category $\cat C$ is $B\cat C:= \abs{N_*\cat C}$. This assignment extends to a functor $B\colon \Cat(\Top)\to \Top$.
 \end{definition}
 
\begin{remark}\label{3thm:BF BG homotopic}
There is a notion of \textit{continuous natural transformation} $\eta\colon F\Rightarrow G$ between two continuous functors $F,G\colon \cat C\to \cat D$, which is equivalent to the data of a continuous functor $\tilde{\eta}\colon \cat C\times [1]\to \cat D$ such that $\tilde\eta(-,0) = F$ and $\tilde\eta(-,1)=G$. Here we regard $[1]$ as the poset category $0<1$. Segal \cite{segal:1968} shows that $B$ sends continuous natural transformations to homotopies, i.e. $BF, BG\colon B\cat C\to B\cat D$ are homotopic maps (via $B\eta$).
\end{remark}

\begin{remark}\label{3rmk:initial gives contractible cl sp}
Note that if a topological category $\cat C$ has an initial object, then the classifying space $B\cat C$ is contractible, induced by the continuous natural transformation between $\id_{\cat C}$ and the constant functor on the initial object. Dually, the same result holds if $\cat C$ has a terminal object.
\end{remark}

One source of topological categories comes from \textit{topologically enriched} categories, where the objects $\cat C_0$ are a discrete set and for each $c_0, c_1\in \cat C_0$ there is a homspace $\cat C(c_0, c_1)$, so then $\cat C_1 = \coprod_{c_0,c_1} \cat C(c_0, c_1)$. For topologically enriched categories, the space of $n$-simplices is\[
\coprod_{(c_0, \dots, c_n)\in \cat C_0^{n+1}} \cat C(c_0, c_1)\times \dots \times \cat C(c_{n-1}, c_n)
\] for $n\geq 1$. The flow category of \cref{2defn:Cf} is an example of such a topological category. 

\begin{remark}\label{rmk:top cat split nerve}
    Suppose that $\cat C$ is a topologically enriched category such that each object $c$ has no non-trivial endomorphisms in $\cat C$. Then the nerve $N_*\cat C$ is a split simplicial space in the sense of \cref{def:splitsim}. Thus \cref{lem:simloc,lem:scw,lem:sep} imply that $B\cat C$ inherits the relevant topological properties (e.g. locally contractible, second countable, homotopy equivalent to a CW complex) from its homspaces.
\end{remark}

In \cref{sec:intermediary}, we will briefly need the notion of a \textit{well-pointed} topological category. 

\begin{definition}
    A topological category $\cat C$ is \textit{well-pointed} if $i\colon \cat C_0\hookrightarrow \cat C_1$ makes $(\cat C_1, \cat C_0)$ an NDR-pair over $\cat C_0\times \cat C_0$ in the sense of \cite[Definition 2.4]{roberts:2022}, using $(s,t)\colon \cat C_1\to \cat C_0\times \cat C_0$.
\end{definition}

The well-pointed assumption is designed to ensure that the nerve $N_*\cat C$ is a good simplicial space in the sense of \cite[Appendix A]{segal:1974}. All of the examples we consider will be covered by the following general case.

\begin{proposition}\label{prop:well ptd}
    Let $\cat C$ be a topological category and suppose $\cat C_1=\cat C_0\sqcup X$ for some space $X$, so that $i\colon \cat C_0\hookrightarrow \cat C_0\sqcup X$ is the canonical inclusion. Then $\cat C$ is well-pointed. 
\end{proposition}\begin{proof}
    Adopting the notation of \cite[Definition 2.4]{roberts:2022}, the map $u$ may be defined as the coproduct of \[\cat C_0\xrightarrow{\Delta} \cat C_0\times \cat C_0 \xrightarrow{\iota_0} I\times \cat C_0\times \cat C_0 \text{ and } X \xrightarrow{(s,t)} \cat C_0\times \cat C_0 \xrightarrow{\iota_1} I \times \cat C_0\times \cat C_0 \] where the map $\iota_u\colon \cat C_0\times \cat C_0 \hookrightarrow I\times \cat C_0\times \cat C_0$ includes at the time-slice $u\in I$. The map $h\colon (\cat C_0\sqcup X)\times I\to \cat C_0\sqcup X$ may be taken to be the projection, and it is straightforward to check these maps make $(\cat C_1,\cat C_0)$ an NDR pair.
\end{proof}

\begin{example}\label{rmk:flow cat is well ptd}
Suppose that $\cat C$ is a topologically enriched category such that each object $c$ has no non-trivial endomorphisms in $\cat C$. Then \cref{prop:well ptd} applies, since $\cat C_1 = \cat C_0\sqcup \coprod_{c_0\neq c_1} \cat C(c_0,c_1)$. Hence every such topological category is well-pointed. This applies, for instance, to the flow category of \cref{2defn:Cf}.
\end{example}

We end with a brief discussion of the twisted arrow category of a topological category, which is used in \cref{sec:intermediary}.

\begin{definition}\label{defn:tw C}
Given a small category $\cat C$, the \textit{twisted arrow category} of $\cat C$, denoted $\tw(\cat C)$, the category with $\ob\tw(\cat C) = \hom\cat C$. A morphism from $a\to b$ to $c\to d$ is given by a pair of morphisms $(f,g)$ that make the square
\begin{center}
    \begin{tikzcd}
    a\arrow[d] \ar[rr, "f"] &&c\arrow[d]\\
    b && d\ar[ll, "g"]
    \end{tikzcd}
\end{center} commute. Composition is given by pasting squares. If $\cat C$ is a topological category, then $\tw\cat C_f$ is a topological category as well, with the morphisms topologized as the iterated pullback $N\cat C_3 = \cat C_1\times_{\cat C_0}\cat C_1\times_{\cat C_0} \cat C_1$. 
\end{definition}

Remarkably, twisting the category in this way does not affect the classifying space, as $B\cat C$ and $B\cat \tw(\cat C)$ are actually homeomorphic.

\begin{corollary}\label{cor:BtwC}
    There is a homeomorphism $B\cat C\cong B\tw(\cat C)$. 
\end{corollary}\begin{proof}
    The key observation is that $N_*\tw\cat C$ is isomorphic, as a simplicial space, to Segal's edgewise subdivision of the nerve, denoted $\sd_*(N\cat C)$; for details on the edgewise subdivision of a simplicial set, see \cite[2.5--2.8]{barwick2013} or \cite[\S 3.1.3]{calle:2020}. We then apply a result of Segal \cite[Appendix 1]{segal:1973} which states that applying edgewise subdivision does not affect geometric realization (up to homeomorphism). Therefore $B\tw\cat C = \abs{N_*\tw\cat C} \cong \abs{\sd_*(N\cat C)}\cong \abs{N_*\cat C} = B\cat C.$
\end{proof}

\section{Topologies on compactified moduli spaces of gradient flow lines}
\label{app:top}

In this appendix, we prove that three topologies on the compactified moduli spaces of gradient flow lines agree. Throughout this section, we fix a closed Riemannian manifold $M$, a Morse function $f\colon M\to \RR$, and two critical points $p,q\in \crit(f)$. We first introduce the relevant topologies on $\overline{\M}(p,q)$ that have appeared in the literature.

\begin{definition}[{Section 1, \cite{cohen/jones/segal:1995}}]
\label{def:heighttop}
Regard a generalized flow line $\gamma\in \overline{\M}(p,q)$ as a continuous function $\gamma\colon [f(q),f(p)]\to M$ which satisfies
    \[ \frac{d\gamma}{dt} = - \frac{\nabla f}{\|\nabla f\|^2} \]
    on the complement of $\crit(f)$. Let $\tau_1$ denote the subspace topology on $\overline{\M}(p,q)$, regarded as a subspace of $\text{Map}([f(q),f(p)],M)$ equipped with the compact-open topology.
\end{definition}

The curve $\gamma$ satisfies $f(\gamma(t))=f(p)+f(q)-t$; we note that there is a sign issue on \cite[p. 3]{cohen/jones/segal:1995}, as the reparametrization with $f(\gamma(t))=t$ instead satisfies the equation $\frac{d\gamma}{dt} = \frac{\nabla f}{\|\nabla f\|^2}$.

\begin{definition}[{\cite[Definition 4.10]{BH:Morse-Bott}}]\label{def:distance fn}
    Let $\tau_2$ denote the topology defined by the distance function
    \[ d(\gamma,\eta):= \sup_{x\in \im\gamma}\inf_{y\in \im\eta}d_M(x,y)+\sup_{y\in\im\eta}\inf_{x\in \im\gamma}d_M(x,y),\qquad \gamma,\eta\in\overline{\M}(p,q), \]
    where $d_M$ is the distance function on $M$. In other words, it is the Gromov--Hausdorff distance between the compact subsets $\im \eta$ and $\im \gamma$.
\end{definition}

\begin{definition}[{\cite[Section 3.2.a]{audin/damian:2014}}]
    Let $\tau_3$ denote the topology generated by the following system of fundamental neighborhoods: Let \[\gamma=(\gamma_0,\dots, \gamma_k)\in \mathcal M(q_0, q_1)\times\dots \times\mathcal M(q_k, q_{k+1})\] be a broken flow line, with $p=q_0$ and $q=q_{k+1}$, and let $\Omega_i= \Omega(q_i)$ be the Morse neighborhood of each critical point. That is, $\Omega_i$ is a neighborhood of $q_i$ which is contained in a local chart $U$ such that \begin{itemize}
        \item $U=\{(x_-,x_+) \mid x_-=(x_1,x_2,\cdots,x_\lambda), x_+=(x_{\lambda+1},\cdots, x_n)\}$ where $\lambda=\mu(q_i)$,
        \item $f=-\|x_-\|^2+\|x_+\|^2+c$ when restricted to $U$,
        \item $\Omega_i=\{ (x_-,x_+)\in U \mid -\varepsilon<-\|x_-\|^2+\|x_+\|^2<\varepsilon \;\;\text{and}\;\; \|x_-\|^2\|x_+\|^2<\varepsilon(\varepsilon+\delta)\}$ for some small enough $\varepsilon,\delta>0$.
    \end{itemize}
    In this construction, $\partial \Omega_i=\partial_+\Omega_i\cup\partial_-\Omega_i\cup \partial_0\Omega_i$, where \[\partial_{\pm}\Omega_i= \{ (x_-,x_+)\in U\mid \|x_\mp\|^2\leq \delta \text{ and } -\|x_-\|^2+\|x_+\|^2=\pm \varepsilon \}\] are two level sets corresponding to the entry and exit sets of $\Omega_i$ respectively (see \cite[Section 2.1.b]{audin/damian:2014}).
    
    Each $\gamma_i$ must exit $\Omega_i$, so we may choose a neighborhood $U_i^-\subseteq \Omega_i$ of the exit point which is contained in its level set. Similarly, $\gamma_i$ must  enter $\Omega_{i+1}$, and we can choose a neighborhood $U_{i+1}^+$ contained in the level set of the entry point. Let $\mathcal{W}(\gamma,U_0^-,\cdots,U_k^+)$ be the set of broken flow lines $\mu=(\mu_1,\dots, \mu_m)$ satisfying the following:\begin{itemize}
    \item[(i)] For each $\mu_j$, there is some $i_j\in\{0,\dots k\}$ such that $s(\mu_j) = q_{i_j}$ and $t(\mu_j) = q_{i_{j+1}}$. We require that $i_0=0$ and $i_j=k$ so that $\mu\in\overline{\mathcal{M}}(p,q)$. 
    \item[(ii)] For $i_j\leq i\leq i_{j+1}-1$, the unbroken flow line $\mu_j$ exits the chart $\Omega_i$ through the interior of $U^-_i$ and enters the chart $\Omega_{i+1}$ through the interior of $U^+_i$.
\end{itemize}
 We say that $\mu$ passes through the ``doors" $U^\pm$. The sets $\mathcal{W}(\gamma,U_0^-,\cdots,U_k^+)$ for all possible choices of $\gamma$ and $U_i^\pm$ form a topological basis of $\overline{\M}(p,q)$, and thus define a topology on $\overline{\M}(p,q)$.
\end{definition}

We now show that these topologies agree. Recall that for two topologies $\tau,\tau^\prime$ on a space $X$ defined by the basis $\beta,\beta^\prime$, $\tau\subseteq \tau^\prime $ if and only if for any $x\in U$ where $U\in \beta$, there is $x\in U^\prime \subseteq U$ where $U^\prime\in \beta^\prime$.

\begin{theorem}
    The three topologies $\tau_1, \tau_2, \tau_3$ on $\overline{\M}(p,q)$ coincide.
\end{theorem}
\begin{proof}
We will show that $\tau_1\subseteq \tau_3\subseteq \tau_2\subseteq \tau_1$.

\textbf{Step 1} ($\tau_2\subseteq \tau_1$): By the definition of the compact-open topology,  $\tau_1$ can also be defined by the distance function
\[ d_1(\gamma,\eta):= \sup_{f(q)\leq t\leq f(p)}d_M(\gamma(t),\eta(t)),\qquad \gamma,\eta\in \overline{\M}(p,q) \]
where $\gamma,\eta$ are reparameterized to satisfy the condition $f(\gamma(t))=f(p)+f(q)-t$. It is then clear that 
\[ d(\gamma,\eta)\leq \sup_{x=\gamma(t_0)} d_M(\gamma(t_0),\eta(t_0))+\sup_{y=\eta(s_0)} d_M(\gamma(s_0),\eta(s_0)) \leq 2d_1(\gamma,\eta) \] for any $\gamma,\eta$. 
Thus $B_{d_1}(\gamma,\varepsilon/2)\subseteq B_d(\gamma,\varepsilon)$ and hence $\tau_2\subseteq \tau_1$.

\textbf{Step 2} ($\tau_3\subseteq \tau_2$): Suppose $\eta\in \mathcal{W}(\gamma,U_0^-,\cdots,U_k^+)$, so $\eta$ passes through all the doors $U_i^{\pm}$. Let $L_i^\pm\supset U_i^\pm$ be the level sets containing these doors. We claim that we can choose small enough $\varepsilon>0$ such that the ball
\[ B_M(\eta,\varepsilon):=\{ x\in M\mid d_M(x,\eta):=\inf_{y\in \im\eta}d_M(x,y)<\varepsilon \}, \]
satisfies
\[ B_M(\eta,\varepsilon)\cap L_i^\pm \subseteq U_i^{\pm}\]
for any $i$ and $\pm$. 
To prove the claim, let $\eta_i^{\pm}=\eta\cap U_i^\pm$ denote the entry/exit points of $\eta$ from the doors. Choose $\varepsilon_i^\pm, \delta_i>0$ such that \[B_M(\eta_i^\pm,\varepsilon_i^\pm)\cap L_i^\pm\subseteq U_i^\pm \text{ and } B_M(x,\delta_i)\cap L_i^{\pm}=\varnothing\] for each $x\in (\im\eta)\setminus B_M(\eta_i^\pm,\varepsilon_i^\pm/2)$.
Set $\varepsilon=\min(\varepsilon_i^\pm/2,\delta_i)$. 

Now, given $y\in B_M(\eta,\varepsilon)\cap L_i^{\pm}$, there exists $w\in \im\eta$ such that $d_M(w,y)<\varepsilon<\delta_i$, hence $w\in B_M(\eta_i^\pm,\varepsilon_i^\pm/2)$ for this $i$.
Then \[d_M(\eta_i^\pm,y)\leq d_M(\eta_i^\pm,w)+d_M(w,y)<\varepsilon_i^\pm/2+\varepsilon_i^\pm/2=\varepsilon_i^\pm,\] so $y\in U_i^\pm$ since $y\in L_i^{\pm}$ by assumption, and we conclude the claim.

Finally, for any $\sigma\in B_d(\eta,\varepsilon)$, we have $\im\sigma\subseteq B_M(\eta,\varepsilon)$ so $\im\sigma\cap L_i^\pm\in U_i^\pm$ hence $\sigma$ passes through the chosen doors. Therefore $\sigma\in \mathcal{W}(\gamma,U_0^-,\cdots,U_k^+)$ and hence there is a containment $B_d(\eta,\varepsilon)\subseteq \mathcal{W}(\gamma,U_0^-,\cdots,U_k^+)$, which shows $\tau_3\subseteq \tau_2$.

\textbf{Step 3} ($\tau_1\subseteq \tau_3$): As we have seen in Step 1, the topology $\tau_1$ is induced by the metric
\[ d_1(\gamma,\eta)= \sup_{f(q)\leq t\leq f(p)}d_M(\gamma(t),\eta(t)),\qquad \gamma,\eta\in \overline{\M}(p,q). \]
Suppose $\gamma\in \overline{\M}(p,q)$ and $\delta>0$. We will produce a neighborhood $\mathcal{W}(\gamma,U_0^-,\cdots, U_k^+)$ of $\gamma$ such that $\mathcal{W}(\gamma,U_0^-,\cdots, U_k^+)\subseteq B_{d_1}(\gamma,\delta)$. 

Suppose that $\gamma=(\gamma_0\circ\cdots\circ \gamma_k)$ for $\gamma_i\in \M(q_{i},q_{i+1})$ and let $\Omega_i$ be a Morse neighborhood of $q_i$ such that $\text{diam}(\Omega_i)=\sup_{x,y\in \Omega_i}d_M(x,y)<\delta$. As in Step 2, let $L_i^+$ and $L_i^-$ be the level sets contained in the boundary of $\Omega_i$ that correspond to the entry and exit sets, respectively, and $\gamma_i^\pm:= \im\gamma\cap L_i^\pm=$ be the entry and exit points of $\gamma$. Let $m_i^\pm=f(p)+f(q)-f(L_i^\pm)\in (f(q),f(p))$.

Now, for a fixed $j$, the generalized flow line $\gamma$ restricts to a map $[m_{j-1}^-, m_j^+]\to M$ that is a solution to the differential equation
\begin{equation}
\label{eqn:ode}
    \frac{d\gamma}{dt}=- \frac{\nabla f}{\| \nabla f \|^2}
\end{equation} By the continuous dependence on initial conditions for the solutions to a differential equation away from critical points, there exists $\varepsilon_j>0$ such that if $\eta$ is also a solution to \ref{eqn:ode} and $d_M(\eta(m_{j-1}^-),\gamma(m_{j-1}^-))<\varepsilon_j$, then $d_M(\eta(t),\gamma(t))<\delta$ for any $t\in [m_{j-1}^-, m_j^+]$. 

Set $\varepsilon=\min_i(\varepsilon_i)$, and define $U_i^\pm=B_M(\gamma_i^\pm,\varepsilon)\cap L_i^\pm$. For any $\eta\in \mathcal{W}(\gamma,U_0^-,U_1^+,\cdots,U_k^+)$, by reparameterization we may assume $\eta$ also satisfies \ref{eqn:ode}. Then for $t\in [f(q),f(p)]$,
\begin{itemize}
    \item If $t\in [m_j^+,m_j^-]$ for some $j$, we have $\eta(t)\in \Omega_j$ and $d_M(\eta(t),\gamma(t))<\delta$ by assumption;
    \item If $t\in [m_{j-1}^-,m_j^+]$ for some $j$, we have $d_M(\eta(t),\gamma(t))<\delta$ by the choice of $\varepsilon_j$.
\end{itemize}
Hence $d_1(\eta,\gamma)=\sup_t d_M(\eta(t),\gamma(t))<\delta$ so $\mathcal{W}(\gamma,U_0^-,\cdots, U_k^+)\subseteq B_{d_1}(\gamma,\delta)$. We conclude that $\tau_1\subseteq \tau_3$.
\end{proof}

\bibliographystyle{alpha}
\bibliography{sample}
\end{document}

%% file: figures/verticaltorus.tex
\begin{tikzpicture}[scale=1.25]
\fill[mid1!60] (0,0) ellipse (1.2 and 1.6);
\draw[dark0] (0,0) ellipse (1.2 and 1.6);
\fill[white] (-0.075,0) ellipse (0.25 and 0.55);
\fill[mid1!60] (0,-0.5) to[bend left= 45] (0, 0.5) -- (0.1, 0.55) -- (-0.5,0.6) -- (-0.5,-0.6) -- (0.1, -0.55) -- cycle;
\draw[dark0] (-0.15,-0.75) to[bend right= 50] (-0.15, 0.75);
\draw[dark0] (0,-0.5) to[bend left= 45] (0, 0.5);


\draw[dark2!80, thick, densely dashed] (0,-1.6) arc (-90:90: 0.21 and 0.5);
\draw[dark2!80, thick] (0,-1.6) arc (270:90: 0.21 and 0.5);
\draw[dark2!80, thick, ->] (-0.2,-1.25) -- (-0.2, -1.3);

\draw[dark2!80, thick, densely dashed] (0,-0.6) to[bend right=80] (0,0.55);
\draw[dark2!80, thick] (0,-0.6) to[bend left=80] (0,0.55);
\draw[dark2!80, thick, ->] (0.325,0.05) -- (0.325, 0.0);

\draw[dark2!80, thick, ->] (0.89,0.65) -- (0.91, 0.6);
\draw[dark2!80, thick] (0,0.55) arc (270:90: 0.21 and 0.5);
\draw[dark2!80, thick, densely dashed] (0,0.55) arc (-90:90: 0.21 and 0.5);

\draw[dark2!80, thick, ->] (-0.19,1.25) -- (-0.21, 1.2);
\draw[dark2!80, thick] (0,1.55) to[bend left=80] (0.2,-1.575);

\draw[dark2!80, thick] (0,1.55) to[out=190,in=140,looseness=1] (-0.125,-1.575);
\draw[dark2!80, thick, ->] (-0.86,0.1) -- (-0.86, 0);

\draw[->] (0,1.55) -- (0,2.25);
\draw (0,-1.65) -- (0,-2);
\draw (0,0.5) -- (0,-0.6);

\draw[->] (0.99, -0.33) -- (1.5,-0.5);
\draw (-1.5,0.5) -- (-1.17, 0.39);
\draw (-0.3,0.1) -- (0.192, -0.064);

\draw[->] (0.195,0.065) -- (-1.5,-0.5);
\draw (1.5,0.5) -- (1.17, 0.39);

\fill[fill=dark1] (0,1.55) circle (1pt) node[above right]{$p_1$};
\fill[fill=dark1] (0,-1.55) circle (1pt) node[below right]{$p_4$};
\fill[fill=dark1] (0,0.55) circle (1pt) node[right]{$p_2$};
\fill[fill=dark1] (0,-0.6) circle (1pt) node[right]{$p_3$};
\end{tikzpicture}

%% file: figures/coincide.tex
\begin{tikzpicture}[scale=1.25]
\draw[dark2, thick] (0,0) -- (6,0) -- (8,1.75) -- (2,1.75) -- cycle;
\fill (1,0) circle (0pt) node[above]{\textcolor{dark2}{$L$}};

\draw[mid4, opacity=0.8] (4,1) ellipse (0.7 and 1.95);
\draw[mid4, opacity=0.8] (4,1) ellipse (0.25 and 1.95);
\draw[mid4, opacity=0.8] (4,1) ellipse (1.3 and 1.95);
\draw[mid4, opacity=0.8] (4,1) ellipse (2 and 2);

\fill[dark1] (4,3) circle (2pt) node[above]{\textcolor{dark1}{$p$}};
\fill[dark1] (4,-1) circle (2pt) node[below]{\textcolor{dark1}{$q$}};

\draw[mid1, thick] (4,1) ellipse (2 and 0.5);
\fill (6.2,1.7) circle (0pt) node[above]{\textcolor{mid1}{$W^u(q)\cap L = W^s(p) \cap L \cong S^1$}};
\fill[mid1] (4.25,0.5) circle (1pt);
\fill[mid1] (5.25,0.6) circle (1pt);
\fill[mid1] (2.7,0.6) circle (1pt);

\fill[mid1] (3.75,1.5) circle (1pt);
\fill[mid1] (4.7,1.475) circle (1pt);
\fill[mid1] (3.3,1.5) circle (1pt);
\end{tikzpicture}

%% file: figures/box.tex
\begin{tikzpicture}[scale=1.25]
\draw[mid1, thick, fill=mid3, fill opacity=0.6] (0,0) -- (5,0) -- (8,2) -- (3,2) -- cycle;
\fill (2,0) circle (0pt) node[below]{\textcolor{mid3}{$N'\cong N\times[-1,1]$}};

\draw[dark3, thick] (1.5,1) -- (6.5,1);
\fill (7,1) circle (0pt) node[right]{\textcolor{dark3}{$N$}};

\draw[->] (3,2) -- (4,2.75) node[below right]{{$y$}};
\draw[dashed] (1.5,1) -- (1.5,4);
\draw[->,] (1.5,4) -- (1.5,5) node[below right]{{$z$}};

\fill (6,5) circle (0pt) node[above]{\textcolor{dark2}{$\nabla f = (0,0,-\delta)$}};
  \def\narrows{8}
  \foreach \i in {1,...,\narrows} {
    \pgfmathsetmacro{\xstart}{\i*0.8}
    \pgfmathsetmacro{\ystart}{1+\i*1/10}

    \pgfmathsetmacro{\dy}{0.75}

    \draw[->, dark2!80] (\xstart,\ystart) -- ++(0,-\dy);
    \draw[->, dark2!80] (0.4+\xstart,1+\ystart) -- ++(0,-0.8*\dy);
    \draw[->, dark2!80] (0.1+\xstart,1.6+\ystart) -- ++(0,-0.6*\dy);
    \draw[->, dark2!80] (0.4+\xstart,2.5+\ystart) -- ++(0,-0.8*\dy);
  }
  
    \draw[->, dark2!80] (7.6,3.2) -- (7.6,2.5);
    \draw[->, dark2!80] (3.2,4.8) -- (3.2,4);
    \draw[->, thin, dark2!80] (4.8,4.8) -- (4.8,4.2);
    \draw[->, dark2!80] (0.3,2) -- (0.3,1.2);
    \draw[->, dark2!80] (4.5,0.8) -- (4.5,0.2);
    
\draw[mid1, thick] (0,3) -- (5,3) -- (8,5) -- (3,5) -- cycle;
\draw[mid1, thick] (0,0) -- (0,3);
\draw[mid1, thick] (5,0) -- (5,3);
\draw[mid1, thick] (8,2) -- (8,5);
\draw[mid1, thick] (3,2) -- (3,5);
\end{tikzpicture}

%% file: figures/slice.tex
\begin{tikzpicture}[scale=0.85]
\draw[mid1, thick, fill=mid3, fill opacity=0.8] (0,0) -- (5,0) -- (8,2) -- (3,2) -- cycle;

\draw[dark3, thick] (1.5,1) -- (6.5,1);

\draw[mid1, thick] (0,3) -- (5,3) -- (8,5) -- (3,5) -- cycle;
\draw[mid1, thick] (0,0) -- (0,3);
\draw[mid1, thick] (5,0) -- (5,3);
\draw[mid1, thick] (8,2) -- (8,5);
\draw[mid1, thick] (3,2) -- (3,5);

\draw[dark2, ->] (3,2) -- (4,2.75) node[below right]{\textcolor{dark3}{$y$}};
\draw[dark2, dashed] (1.5,1) -- (1.5,4);
\draw[dark2, ->,] (1.5,4) -- (1.5,5) node[below right]{\textcolor{dark3}{$z$}};

\draw[mid5, thick, fill=mid5, fill opacity=0.2] (2.5,0) -- (2.5,3) -- (5.5, 5) -- (5.5, 2) -- cycle;
\fill[dark3] (4,1) circle (2pt) node[below right]{\textcolor{dark3}{$x_0\in N$}};

\draw[mid5, thick, fill=mid5, fill opacity=0.2] (9,1) -- (9,3) -- (12, 3) -- (12, 1) -- cycle;
\fill[dark3] (10.5, 3) circle (1pt) node[above]{\textcolor{dark3}{$x=x_0$}};
\fill[dark3] (10.5, 1) circle (1pt) node[below]{\textcolor{dark3}{$x=x_0$}};
\draw[dark3, ->] (10.5, 3) -- (10.5, 1.6) node[above right]{\textcolor{dark3}{$\gamma$}};
\draw[dark3] (10.5, 2) -- (10.5, 1);

\draw[mid5, ->] (9.5, 2.5) -- (9.5, 2);
\draw[mid5, ->] (10, 2) -- (10, 1.25);
\draw[mid5, ->] (10.25, 2.8) -- (10.25, 2.2);
\draw[mid5, ->] (11.75, 2) -- (11.75, 1.4);
\draw[mid5, ->] (11, 2) -- (11, 1.25);
\draw[mid5, ->] (11.25, 2.8) -- (11.25, 2.2);

\draw[mid5, thick, fill=mid5, fill opacity=0.2] (14,1) -- (14,3) -- (17, 3) -- (17, 1) -- cycle;
\fill[dark3] (15.5, 3) circle (1pt) node[above]{\textcolor{dark3}{$x=x_0$}};
\draw[dark3] (15.5,3) to[out=270,in=80,looseness=1] (15, 1); \fill[dark3] (15, 1) circle (1pt) node[below]{\textcolor{dark3}{$(x_0, \lambda(x_0), 0)$}};

\draw[mid5, ->] (14.5, 2.5) -- (14.45, 2);
\draw[mid5, ->] (15, 2) -- (14.65, 1.3);
\draw[mid5, ->] (15.25, 2.8) -- (15, 2.2);
\draw[mid5, ->] (15.8, 2) -- (15.5, 1.25);
\draw[mid5, ->] (16.25, 2.8) -- (16.1, 
2.2);
\draw[mid5, ->] (16.75, 2) -- (16.75, 1.4);
\draw[mid5, ->] (16.35, 1.8) -- (16.25, 1.4);

\draw[dark1, thick, ->] (7,1) to[out=270,in=200,looseness=1] (9, 0.5) node[below left]{\textcolor{dark1}{slice at $x_0$}};
\draw[dark1, thick, ->] (12,3.5) to[out=60,in=90,looseness=1] (14, 3.25);
\draw[dark3] (13, 4) circle (0pt)  node[above]{\textcolor{dark1}{twisting}};
\end{tikzpicture}

%% file: figures/S2S1.tex
\begin{tikzpicture}[scale=1.25]
\fill[mid0!60] (0,0) ellipse (1.6 and 0.9);
\draw[mid0, thick] (0,0) ellipse (1.6 and .9);
\fill[white] (0,-0.05) ellipse (0.5 and 0.225);
\fill[mid0!60] (-0.75,0.1) to[bend right= 30] (0.75,0.1) to (0.75, -0.5) to (-0.75, -0.5) -- cycle;
\draw[mid0, thick] (-0.75,0.1) to[bend right= 30] (0.75,0.1);
\draw[mid0, thick] (-0.5,0) to[bend left= 40] (0.5,0);

\draw[mid0!80, thick] (0,-.9) arc (270:90:0.21 and .38325);
\draw[mid0!80, densely dashed, thick] (0,-.9) arc (-90:90:0.21 and .38325);
\draw[mid0, thin] (-0.1,-0.8) -- (-0.2, -1.2) node[below] {\textcolor{mid0}{$S^1$}};

\draw[->] (0,-0.1) -- (0,1.5) node[right] {$y$};

\draw (1.375,0.4583) -- (1.8,0.6);
\draw[->] (-1.2, -0.4) -- (-2.1,-0.7);

\draw[->] (1.6,0) -- (2.5,0) node[below left] {$\RR^3$};
\draw (-2,0) -- (-1.6,0);

\draw (-0.45,0) -- (0.5,-0);
\draw (-0.285,-0.095) -- (0.45,0.15);

\draw[dark1, thick] (0,0.25) ellipse (1 and .5);
\draw[dark1, thin] (1,0.3) -- (1, 1.5) node[above] {\textcolor{dark1}{$S^2$}};

\fill[fill=dark0] (1.6,0) circle (1pt) node[above right]{$p_1$};
\fill[fill=dark0] (0.5,0) circle (1pt) node[above right]{$p_2$};
\fill[fill=dark0] (-0.5,0) circle (1pt) node[above]{$p_3$};
\fill[fill=dark0] (-1.6,0) circle (1pt) node[above left]{$p_4$};

\end{tikzpicture}

%% file: figures/y=0.tex
\begin{tikzpicture}[scale=1.25]

\draw[dark0, thick] (0,0) ellipse (2 and 2);
\draw[dark0, thick] (0,0) ellipse (1 and 1);
\draw[dark0] (-2,0) arc (-180:0: 2 and 0.5);
\draw[dark0, dashed] (-2,0) arc (180:0: 2 and 0.5);
\draw[dark0] (-1,0) arc (-180:0: 1 and 0.25);
\draw[dark0, dashed] (-1,0) arc (180:0: 1 and 0.25);

\draw[dark2, thick] (0,1) arc (90:270: 0.5 and 1);
\draw[dark2, dashed, thick] (0,1) arc (90:-90: 0.5 and 1);
\fill[fill=dark0] (0,1) circle (0pt) node[above]{\textcolor{dark2}{$N$}};

\draw[mid1, thick] (-1,0) arc (-180:0: 1 and 0.1);
\draw[mid1, thick] (-1,0) arc (-180:0: 1 and 0.8);
\draw[mid1, thick] (-1,0) arc (-180:0: 1 and 0.1);
\draw[mid1, thick] (-1,0) arc (-180:0: 1 and 0.5);

\draw[mid1, thick] (-1,0) arc (180:0: 1 and 0.9);
\draw[mid1, thick] (-1,0) arc (180:0: 1 and 0.7);
\draw[mid1, thick] (-1,0) arc (180:0: 1 and 0.5);

\draw[mid1, thick, ->] (0,0.9) -- (-0.1, 0.9);
\draw[mid1, thick, ->] (0,0.5) -- (-0.1, 0.5);
\draw[mid1, thick, ->] (0,-0.5) -- (-0.1, -0.5);
\draw[mid1, thick, ->] (0,-0.1) -- (-0.1, -0.1);

\fill[fill=dark0] (-2,0) circle (1.5pt) node[left]{$p_4$};
\fill[fill=dark0] (-1,0) circle (1.5pt) node[left]{$p_3$};
\fill[fill=dark0] (1,0) circle (1.5pt) node[right]{$p_2$};
\fill[fill=dark0] (2,0) circle (1.5pt) node[right]{$p_1$};
\end{tikzpicture}

%% file: figures/f=0.tex
\begin{tikzpicture}[scale=1.25]
\fill[mid1!60] (0,0) ellipse (1.2 and 1.6);
\draw[mid1, thick] (0,0) ellipse (1.2 and 1.6);
\fill[white] (-0.075,0) ellipse (0.25 and 0.55);
\fill[mid1!60] (0,-0.5) to[bend left= 45] (0, 0.5) -- (0.1, 0.55) -- (-0.5,0.6) -- (-0.5,-0.6) -- (0.1, -0.55) -- cycle;
\draw[mid1, thick] (-0.15,-0.75) to[bend right= 50] (-0.15, 0.75);
\draw[mid1, thick] (0,-0.5) to[bend left= 45] (0, 0.5);

\draw[mid1] (-0.2,0.25) arc (10:-180: 0.5 and 0.15);
\draw[mid1, dashed] (-0.2,0.25) arc (10:180: 0.5 and 0.15);


\draw[dark2, thick] (0,0.65) arc (90:270: 0.3 and 0.6);
\draw[dark2, thick, dashed] (0,0.65) arc (90:-270: 0.3 and 0.6);

\draw[->] (0,1.55) -- (0,2.25);
\draw (0,-1.62) -- (0,-2);
\draw (0,0.5) -- (0,-0.6);

\draw[->] (0.99, -0.33) -- (1.5,-0.5);
\draw (-1.5,0.5) -- (-1.17, 0.39);
\draw (-0.3,0.1) -- (0.192, -0.064);

\draw[->] (0.195,0.065) -- (-1.5,-0.5);
\draw (1.5,0.5) -- (1.17, 0.39);

\fill[fill=dark1] (0.25,1.5) circle (0pt) node[above right]{$f=0$};
\fill[fill=dark2] (0.25,-0.5) circle (0pt) node[above right]{\textcolor{dark2}{$N$}};
\fill[fill=mid1] (1,-1.5) circle (0pt) node[right]{\textcolor{mid1}{$L$}};
\end{tikzpicture}

%% file: figures/onL2.tex
\begin{tikzpicture}[scale=1.25]
\draw[dark2, thick] (0,0) -- (10.1,0);
\fill (10.1,0) circle (0pt) node[right]{\textcolor{dark2}{$N= W_g^s(p_3)\cap L$}};

\draw[mid4, thick] (1,0) arc (180:0: 0.5 and 0.5);
\draw[mid4, thick] (3,0) arc (180:0: 0.5 and 0.5);
\draw[mid4, thick] (5,0) arc (180:0: 0.5 and 0.5);
\draw[mid4, thick] (7,0) arc (180:0: 0.5 and 0.5);
\draw[mid4, thick] (9,0) arc (180:0: 0.5 and 0.5);
\fill (5,0.5) circle (0pt) node[]{\textcolor{mid4}{$I^+$}};

\draw[mid2, thick] (1,0) arc (0:-180: 0.5 and 0.5);
\draw[mid2, thick] (3,0) arc (0:-180: 0.5 and 0.5);
\draw[mid2, thick] (5,0) arc (0:-180: 0.5 and 0.5);
\draw[mid2, thick] (7,0) arc (0:-180: 0.5 and 0.5);
\draw[mid2, thick] (9,0) arc (0:-180: 0.5 and 0.5);
\fill (5,-0.5) circle (0pt) node[]{\textcolor{mid2}{$I^-$}};

\fill[dark0] (8,0) circle (2pt) node[below left]{\textcolor{dark0}{$Z$}};
\fill[dark0] (7,0) circle (2pt);
\fill[dark0] (6,0) circle (2pt);
\fill[dark0] (5,0) circle (2pt);
\fill[dark0] (4,0) circle (2pt);
\fill[dark0] (3,0) circle (2pt);
\fill[dark0] (2,0) circle (2pt);
\fill[dark0] (1,0) circle (2pt);

\draw[->, dark0] (0.5,-0.2) -- (0.5,1.5) node[left]{\textcolor{dark0}{$y$}}; 
\fill (1,1) circle (0pt) node[]{\textcolor{dark0}{$L$}};

\fill (8,1) circle (0pt) node[right]{\textcolor{mid2}{$W_g^u(p_2)\cap L=I^+\cup I^-\cup Z$}};

\end{tikzpicture}

%% file: figures/recall.tex
\begin{tikzpicture}[scale=1.25]

\draw (-1,0) -- (-1,3);
\draw (-0.9,0) -- (-1.1, 0) node[left] {-1};
\draw (-0.9,1.5) -- (-1.1, 1.5) node[left] {0};
\draw (-0.9,3) -- (-1.1, 3) node[left] {1};
\fill (0,4) circle (0pt) node[left] {{$S^2\times S^1$}};

\draw[dark1] (0,0) -- (0,3);
\draw[dark1] (8,2) -- (8,5);
\draw[dark1] (3,2) -- (3,5);
\draw[dark1, fill opacity = 0.1, fill=dark1] (0,0) -- (5,0) -- (8,2) -- (3,2) -- cycle;
\fill (4,4) circle (0pt) node[above]{\textcolor{dark1!80}{$\bndry P_1$}};
\fill (4,0.5) circle (0pt) node[above]{\textcolor{dark1!80}{$\bndry P_0$}};

\draw[mid1, thick, fill=mid3, fill opacity=0.6] (0,1.5) -- (5,1.5) -- (8,3.5) -- (3,3.5) -- cycle;
\fill (7,3.5) circle (0pt) node[above]{\textcolor{mid3}{$L$}};

\draw[dark2, thick] (1.5,2.5) -- (6.5,2.5);
\fill (6,2.5) circle (0pt) node[above right]{\textcolor{dark2}{$N$}};

\draw[dark1, fill opacity = 0.1, fill=dark1] (0,3) -- (5,3) -- (8,5) -- (3,5) -- cycle;
\draw[dark1] (5,0) -- (5,3);
\fill (6,0.5) circle (0pt) node[right]{\textcolor{dark1}{$L'\cong L\times I$}};

\end{tikzpicture}

%% file: figures/M24.tex
\begin{tikzpicture}

\draw[mid4, thick] (1,-1.75) arc (180:0: 0.5 and 0.5);
\draw[mid4, thick] (3,-1.75) arc (180:0: 0.5 and 0.5);
\draw[mid4, thick] (5,-1.75) arc (180:0: 0.5 and 0.5);
\draw[mid4, thick] (7,-1.75) arc (180:0: 0.5 and 0.5);
\draw[mid4, thick] (9,-1.75) arc (180:0: 0.5 and 0.5);

\foreach \x in {1,...,10}
\fill[mid4] (\x,-1.75) circle (1pt);

\draw[mid2, thick] (1,-2.25) arc (0:-180: 0.5 and 0.5);
\draw[mid2, thick] (3,-2.25) arc (0:-180: 0.5 and 0.5);
\draw[mid2, thick] (5,-2.25) arc (0:-180: 0.5 and 0.5);
\draw[mid2, thick] (7,-2.25) arc (0:-180: 0.5 and 0.5);
\draw[mid2, thick] (9,-2.25) arc (0:-180: 0.5 and 0.5);

\foreach \x in {0,...,9}
\fill[mid2] (\x,-2.25) circle (1pt);

\fill (10,-1.75) circle (0pt) node[right]{\textcolor{mid4}{$\dots$}};
\fill (1,-1.75) circle (0pt) node[left]{\textcolor{mid4}{$\dots$}};
\fill (11,-1.5) circle (0pt) node[right]{\textcolor{mid4}{$I^+ \cup Z$}};

\fill (11.5,-2) circle (0pt) node[right]{\textcolor{mid3}{$\sqcup$}};

\fill (11,-2.5) circle (0pt) node[right]{\textcolor{mid2}{$I^- \cup Z$}};
\fill (9,-2.25) circle (0pt) node[right]{\textcolor{mid2}{$\dots$}};
\fill (0,-2.25) circle (0pt) node[left]{\textcolor{mid2}{$\dots$}};

\end{tikzpicture}

%% file: figures/M13.tex
\begin{tikzpicture}

\draw[mid4, thick] (1,-1.75) --(2,-1.75);
\draw[mid4, thick] (3,-1.75) -- (4,-1.75);
\draw[mid4, thick] (5,-1.75) -- (6,-1.75);
\draw[mid4, thick] (7,-1.75) -- (8,-1.75);
\draw[mid4, thick] (9,-1.75) -- (10,-1.75);

\foreach \x in {1,...,10}
\fill[mid4] (\x,-1.75) circle (1pt);

\draw[mid2, thick] (1,-2.25) -- (0,-2.25);
\draw[mid2, thick] (3,-2.25) -- (2,-2.25);
\draw[mid2, thick] (5,-2.25) -- (4,-2.25);
\draw[mid2, thick] (7,-2.25) -- (6,-2.25);
\draw[mid2, thick] (9,-2.25) -- (8,-2.25);

\foreach \x in {0,...,9}
\fill[mid2] (\x,-2.25) circle (1pt);

\fill (10,-1.75) circle (0pt) node[right]{\textcolor{mid4}{$\dots$}};
\fill (1,-1.75) circle (0pt) node[left]{\textcolor{mid4}{$\dots$}};
\fill (11,-1.5) circle (0pt) node[right]{\textcolor{mid4}{$I^+ \cup Z$}};

\fill (11.5,-2) circle (0pt) node[right]{\textcolor{mid3}{$\sqcup$}};

\fill (11,-2.5) circle (0pt) node[right]{\textcolor{mid2}{$I^- \cup Z$}};
\fill (9,-2.25) circle (0pt) node[right]{\textcolor{mid2}{$\dots$}};
\fill (0,-2.25) circle (0pt) node[left]{\textcolor{mid2}{$\dots$}};

\end{tikzpicture}

%% file: figures/Gammas.tex
\begin{tikzpicture}
\fill[dark2, opacity = 0.5] (0,0) -- (10,0) -- (10,1) -- (0,1) -- cycle;
\draw[dark2, thick] (0,0) -- (10,0);
\draw[dark2, dashed, thick] (0,0.97) -- (10,0.97);
\fill (5,0) circle (0pt) node[above]{\textcolor{dark2}{$L\setminus N\times (0,1)$}};

\fill[dark2, opacity = 0.5] (0,0) -- (10,0) -- (10,-1) -- (0,-1) -- cycle;
\draw[dark2, thick] (0,-1) -- (10,-1);
\draw[dark2, thick, fill=white] (1,-1.03) arc (180:0: 0.5 and 0.5);
\draw[dark2, thick, fill=white] (3,-1.03) arc (180:0: 0.5 and 0.5);
\draw[dark2, thick, fill=white] (5,-1.03) arc (180:0: 0.5 and 0.5);
\draw[dark2, thick, fill=white] (7,-1.03) arc (180:0: 0.5 and 0.5);
\draw[dark2, thick, fill=white] (9,-1.03) arc (180:0: 0.5 and 0.5);

\draw[mid4, thick, fill=mid4, fill opacity = 0.5] (1,-1.75) arc (180:0: 0.5 and 0.5);
\draw[mid4, thick] (1,-1.75)--(2, -1.75);
\draw[mid4, thick, fill=mid4, fill opacity = 0.5] (3,-1.75) arc (180:0: 0.5 and 0.5);
\draw[mid4, thick] (3,-1.75)--(4, -1.75);
\draw[mid4, thick, fill=mid4, fill opacity = 0.5] (5,-1.75) arc (180:0: 0.5 and 0.5);
\draw[mid4, thick] (5,-1.75)--(6, -1.75);
\draw[mid4, thick, fill=mid4, fill opacity = 0.5] (7,-1.75) arc (180:0: 0.5 and 0.5);
\draw[mid4, thick] (7,-1.75)--(8, -1.75);
\draw[mid4, thick, fill=mid4, fill opacity = 0.5] (9,-1.75) arc (180:0: 0.5 and 0.5);
\draw[mid4, thick] (9,-1.75)--(10, -1.75);

\draw[mid2, thick, fill=mid2, fill opacity = 0.5] (1,-2.25) arc (0:-180: 0.5 and 0.5);
\draw[mid2, thick] (0,-2.25)--(1, -2.25);
\draw[mid2, thick, fill=mid2, fill opacity = 0.5] (3,-2.25) arc (0:-180: 0.5 and 0.5);
\draw[mid2, thick] (2,-2.25)--(3, -2.25);
\draw[mid2, thick, fill=mid2, fill opacity = 0.5] (5,-2.25) arc (0:-180: 0.5 and 0.5);
\draw[mid2, thick] (4,-2.25)--(5, -2.25);
\draw[mid2, thick, fill=mid2, fill opacity = 0.5] (7,-2.25) arc (0:-180: 0.5 and 0.5);
\draw[mid2, thick] (6,-2.25)--(7, -2.25);
\draw[mid2, thick, fill=mid2, fill opacity = 0.5] (9,-2.25) arc (0:-180: 0.5 and 0.5);
\draw[mid2, thick] (8,-2.25)--(9, -2.25);

\fill[dark2, opacity = 0.5] (0,-3) -- (10,-3) -- (10,-4) -- (0,-4) -- cycle;
\draw[dark2, thick] (0,-3) -- (10,-3);
\draw[dark2, thick, fill=white] (1,-2.97) arc (0:-180: 0.5 and 0.5);
\draw[dark2, thick, fill=white] (3,-2.97) arc (0:-180: 0.5 and 0.5);
\draw[dark2, thick, fill=white] (5,-2.97) arc (0:-180: 0.5 and 0.5);
\draw[dark2, thick, fill=white] (7,-2.97) arc (0:-180: 0.5 and 0.5);
\draw[dark2, thick, fill=white] (9,-2.97) arc (0:-180: 0.5 and 0.5);

\fill[dark2, opacity = 0.5] (0,-4) -- (10,-4) -- (10,-5) -- (0,-5) -- cycle;
\draw[dark2, thick] (0,-4) -- (10,-4);
\draw[dark2, dashed, thick] (0,-4.97) -- (10,-4.97);
\fill (5,-4) circle (0pt) node[below]{\textcolor{dark2}{$L\setminus N\times (0,1)$}};

\draw[white, fill=white] (0.2,1) -- (0.2,-5) -- (-1, -5) -- (-1,1) -- cycle;
\draw[white, fill=white] (9.8,1) -- (9.8,-5) -- (11, -5) -- (11,1) -- cycle;
\fill (10,0) circle (0pt) node[right]{\textcolor{dark2}{$\Gamma_0$}};
\fill (5,0) circle (0pt) node[below]{\textcolor{dark2}{$\Gamma_+$}};
\fill (10,-1.5) circle (0pt) node[right]{\textcolor{mid4}{$\Gamma_1$}};
\fill (10,-2.5) circle (0pt) node[right]{\textcolor{mid2}{$\Gamma_2$}};
\fill (5, -4) circle (0pt) node[above]{\textcolor{dark2}{$\Gamma_-$}};
\fill (10,-4) circle (0pt) node[right]{\textcolor{dark2}{$\Gamma_0$}};
\end{tikzpicture}